\documentclass[11pt,reqno]{amsart}

\usepackage{color}
\usepackage[colorlinks=true, allcolors=blue]{hyperref}
\usepackage{amsmath, amssymb, amsthm}
\usepackage{mathrsfs}
\usepackage{mathtools}
\usepackage[noabbrev,capitalize,nameinlink]{cleveref}
\crefname{equation}{}{}
\usepackage{fullpage}
\usepackage[noadjust]{cite}
\usepackage{graphics}
\usepackage{pifont}
\usepackage{tikz}
\usepackage{bbm}
\usepackage[T1]{fontenc}

\usetikzlibrary{arrows.meta}

\usepackage{environ}
\usepackage{framed}
\usepackage{url}
\usepackage[linesnumbered,ruled,vlined]{algorithm2e}
\usepackage[noend]{algpseudocode}
\usepackage[labelfont=bf]{caption}
\usepackage{cite}
\usepackage{framed}
\usepackage[framemethod=tikz]{mdframed}
\usepackage{appendix}
\usepackage{graphicx}
\usepackage[textsize=tiny]{todonotes}
\usepackage{tcolorbox}
\usepackage{enumerate}
\allowdisplaybreaks[1]
\usepackage{enumerate}

\crefname{algocf}{Algorithm}{Algorithms}

\crefname{equation}{}{} 
\AtBeginEnvironment{appendices}{\crefalias{section}{appendix}} 

\usepackage[color,final]{showkeys} 

\colorlet{refkey}{orange!20}
\colorlet{labelkey}{blue!30}

\crefname{algocf}{Algorithm}{Algorithms}

\numberwithin{equation}{section}
\newtheorem{theorem}{Theorem}[section]
\newtheorem{proposition}[theorem]{Proposition}
\newtheorem{lemma}[theorem]{Lemma}

\crefname{claim}{Claim}{Claims}

\newtheorem{corollary}[theorem]{Corollary}
\newtheorem{conjecture}[theorem]{Conjecture}
\newtheorem*{question*}{Question}

\theoremstyle{definition}
\newtheorem{definition}[theorem]{Definition}

\newtheorem*{definition*}{Definition}

\newtheoremstyle{bfnote}%
{}{}%
{}{}%
{\bfseries}{.}%
{ }%
{\thmname{#1}\thmnumber{ #2}\thmnote{ (#3)}}
\theoremstyle{bfnote}
\newtheorem*{remark}{Remark}

\newcommand{\norm}[1]{\bigg\lVert#1\bigg\rVert}
\newcommand{\snorm}[1]{\lVert#1\rVert}

\newcommand{\mb}{\mathbb}

\newcommand{\mbm}{\mathbbm}
\newcommand{\mc}{\mathcal}

\newcommand{\ol}{\overline}
\newcommand{\on}{\operatorname}

\newcommand{\wt}{\widetilde}

\allowdisplaybreaks
\title{Singularity of Discrete Random Matrices}

\author[A1]{Vishesh Jain}
\address{Department of Statistics, Stanford University, Stanford, CA 94305, USA}
\email{visheshj@stanford.edu}

\author[A2]{Ashwin Sah}
\author[A3]{Mehtaab Sawhney}
\address{Department of Mathematics, Massachusetts Institute of Technology, Cambridge, MA 02139, USA}
\email{\{asah,msawhney\}@mit.edu}

\begin{document}

\begin{abstract}
Let $\xi$ be a non-constant real-valued random variable with finite support, and let $M_{n}(\xi)$ denote an $n\times n$ random matrix with entries that are independent copies of $\xi$. 

For $\xi$ which is not uniform on its support, we show that
\begin{align*}
\mb{P}[M_{n}(\xi)\text{ is singular}] &= \mb{P}[\text{zero row or column}] + (1+o_n(1))\mb{P}[\text{two equal (up to sign) rows or columns}],
\end{align*}
thereby confirming a folklore conjecture. 

As special cases, we obtain:
\begin{itemize}
\item For $\xi = \on{Bernoulli}(p)$ with fixed $p \in (0,1/2)$, 
\[\mb{P}[M_{n}(\xi)\text{ is singular}] = 2n(1-p)^{n} + (1+o_n(1))n(n-1)(p^2 + (1-p)^2)^{n},\]
which determines the singularity probability to \emph{two} asymptotic terms. Previously, no result of such precision was available in the study of the singularity of random matrices. 

\item For $\xi = \on{Bernoulli}(p)$ with fixed $p \in (1/2,1)$,
\[\mb{P}[M_{n}(\xi)\text{ is singular}] = (1+o_n(1))n(n-1)(p^2 + (1-p)^2)^{n}.\]
Previously, only the much weaker upper bound of $(\sqrt{p} + o_n(1))^{n}$ was known due to the work of Bourgain-Vu-Wood.
\end{itemize}
For $\xi$ which is uniform on its support:
\begin{itemize}
\item We show that
\begin{align*}
\mb{P}[M_{n}(\xi)\text{ is singular}] &= (1+o_n(1))^{n}\mb{P}[\text{two rows or columns are equal}].
\end{align*}
\item Perhaps more importantly, we provide a sharp analysis of the contribution of the `compressible' part of the unit sphere to the lower tail of the smallest singular value of $M_{n}(\xi)$.
\end{itemize}
\end{abstract}

\maketitle

\section{Introduction}\label{sec:introduction}
Let $M_n(\xi)$ be an $n\times n$ random matrix, each of whose entries is an independent copy of a random variable $\xi$. We will restrict attention to when $\xi$ is a real-valued random variable whose support is finite and contains at least two points (which we call \emph{discrete}). What is the probability that $M_n(\xi)$ is singular? This question, which has been studied since the 1960s, has attracted considerable attention over the years. A well-known folklore conjecture is that the dominant contribution to the probability of singularity is from the events that a row or column is zero, or that two rows or two columns are equal (possibly up to a sign). In order to facilitate discussion, let us introduce some notation. For a vector $v\in \mb{R}^{n}$, we define the event 
\[\mc{E}_{v} := \{M_n(\xi)v = 0\}.\]
We will also denote the canonical basis vectors of $\mb{R}^{n}$ by $e_1,\dots,e_n$.
Then, the aforementioned conjecture may be stated as follows. 
\begin{conjecture}
\label{conj:exact}
Let $\xi$ be a discrete random variable, and let $M_n(\xi)$ be an $n\times n$ random matrix whose entries are independent copies of $\xi$. Then
\[\mb{P}[M_n(\xi)\emph{ is singular}] = (1+o_n(1))\bigg(2n\mb{P}[\mc{E}_{e_1}] + n(n-1)\mb{P}[\mc{E}_{e_1 - e_2}] + n(n-1)\mb{P}[\mc{E}_{e_1 + e_2}]\bigg).\]
\end{conjecture}


In this paper, as our first main result, we confirm a stronger version of \cref{conj:exact} for \emph{all} discrete distributions which are not uniform on their support. Let $s_n(M_n)$ denote the least singular value of an $n\times n$ matrix $M_n$; recall that $s_{n}(M_n) = \inf_{x\in \mb{S}^{n-1}}\snorm{M_n x}_{2}$, where $\mb{S}^{n-1}$ denotes the unit sphere in $\mb{R}^{n}$ and $\snorm{\cdot}_{2}$ denotes the standard Euclidean norm on $\mb{R}^{n}$.

\begin{theorem}\label{thm:main-exact}
Let $\xi$ be a discrete random variable which is not uniform on its support. There exist $c_\xi, C_\xi > 0$ so that for all sufficiently large $n$, and for all $t\ge 0$,
\[\mb{P}[s_n(M_n(\xi))\le t/\sqrt{n}]\le C_\xi t + 2n\mb{P}[\mc{E}_{e_1}] + (1+O(\exp(-c_\xi n))) \bigg(n(n-1) \mb{P}[\mc{E}_{e_1+e_2}] + n(n-1) \mb{P}[\mc{E}_{e_1-e_2}] \bigg).\]
\end{theorem}

By applying \cref{thm:main-exact} with $t = 0$ for the upper bound, and considering the probability that a row or column is zero, or that two rows or two columns are the same (up to a sign) for the lower bound (cf.~the corresponding calculation in \cite[Section~3.2]{LT20}), we thus establish the following strengthening of \cref{conj:exact} for discrete distributions which are not uniform on their support. 
\begin{corollary}
\label{cor:exact}
Let $\xi$ be a discrete random variable which is not uniform on its support. There exists $c_\xi > 0$ such that 
\[\mb{P}[M_n(\xi)\emph{ is singular}] = 2n\mb{P}[\mc{E}_{e_1}] + (1+O(\exp(-c_\xi n)))\bigg(n(n-1) \mb{P}[\mc{E}_{e_1+e_2}] + n(n-1) \mb{P}[\mc{E}_{e_1-e_2}] \bigg).\]
In particular, for fixed $p \in (0,1), p \neq 1/2$, there exists $c_p > 0$ for which
\[\mb{P}[M_{n}(\on{Ber}(p)) \emph{ is singular}] = 2n(1-p)^{n} + (1+O(\exp(-c_pn)))n(n-1)(p^2 + (1-p)^2)^n.\]
\end{corollary}
\begin{remark}
Let us discuss the implications of this corollary for $\xi = \on{Ber}(p)$.
\begin{itemize}
\item For fixed $p \in (1/2, 1)$, \cref{cor:exact} provides the leading term in the asymptotic expansion of the singularity probability of $M_{n}(\on{Ber}(p))$. Prior to this work, even the correct value of $\lim_{n\to \infty}n^{-1}\log\mb{P}[M_{n}(\on{Ber}(p)) \text{ is singular})]$ had not been determined; compared to the true value of $(p^{2} + (1-p)^{2})$ for this quantity, the previous best-known result of Bourgain, Vu, and Wood \cite{BVW10} provides a weaker upper bound of $\sqrt{p}$. The reason that the case $p \in (1/2, 1)$ is more challenging than $p \in (0,1/2)$ (treated in \cite{BR18, LT20, Hua20}, see the discussion below) is that in the former case, the dominant contribution to the probability of singularity comes from the event of \emph{two} rows or columns being equal to each other, whereas in the latter case, the dominant contribution comes from the much simpler event of a single row or column being zero. 

\item However, even for fixed $p \in (0,1/2)$, we obtain more precise results than those in the recent work of Litvak and Tikhomirov \cite{LT20} as well as \cite{Tik20,JSS20discrete1}. Namely, \cref{cor:exact} provides the first \emph{two} terms in the asymptotic expansion of the singularity probability of $M_{n}(\on{Ber}(p))$; a result of this precision was not available before in any context.
\end{itemize}
\end{remark}

For general discrete distributions, we determine the value of $\lim_{n \to \infty}n^{-1}\log{\mb{P}[M_{n}(\xi)\text{ is singular}]}$. The only case not covered by \cref{thm:main-exact} is that of uniform distributions, which we handle with a non-exact main term.

\begin{theorem}\label{thm:main-approximate}
Let $\xi$ be a discrete random variable. There exists $C_\xi > 0$ such that for any fixed $\epsilon > 0$ and for all sufficiently large $n$ and all $t\ge 0$, 
\[\mb{P}[s_n(M_n)\le t/\sqrt{n}]\le C_\xi t + 2n\mb{P}[\mc{E}_{e_1}] +  (1+\epsilon)^n\mb{P}[\mc{E}_{e_1-e_2}].\]
\end{theorem}
\begin{remark}
For non-uniform discrete distributions, \cref{thm:main-exact} is strictly stronger.
\end{remark}

\subsection{Previous work}\label{sub:previous-work}
Let us put \cref{thm:main-exact,thm:main-approximate} in the context of known results. For convenience, we will use $q_{n}(\xi)$ to denote $\mb{P}[M_n(\xi)\text{ is singular}]$. The work of Koml\'os \cite{Kom67} was the first to show that $q_{n}(\on{Ber}(1/2)) = o_n(1)$. Much later, an exponential bound on $q_{n}(\on{Ber}(1/2))$ was obtained by Kahn, Koml\'os, and Szemer\'edi \cite{KKS95}. Subsequently, the base of the exponent was improved to $0.939$ and $3/4 + o_n(1)$ in a series of works by Tao and Vu  \cite{TV06, TV07}, and later to $1/\sqrt{2} + o_n(1)$ by Bourgain, Vu, and Wood \cite{BVW10}. Finally, a truly breakthrough result of Tikhomirov \cite{Tik20} in 2018 established that $q_n(\on{Ber}(p)) = (1-p+o_n(1))^{n}$ for fixed $p \in (0,1/2]$. As mentioned earlier, for fixed $p \in (1/2, 1)$, the analogous result was not known prior to this work. 

\cref{conj:exact} has been most accessible for sparse Bernoulli distributions, in which case, the right hand side simplifies considerably to $(1+o_n(1))\cdot 2n\mb{P}[\mc{E}_{e_1}]$. Here, by the Bernoulli distribution with parameter $p$, which we will henceforth denote by $\on{Ber}(p)$, we mean the two point distribution which attains the value $1$ with probability $p$ and the value $0$ with probability $1-p$. Basak and Rudelson \cite{BR18} confirmed the conjecture for $\xi = \on{Ber}(p_n)$ for $p_n$ in a certain range of sparsity limited to $n^{-1}\ln n - \omega_n(n^{-1})\le p_n \le n^{-1}\ln n + o_n(n^{-1}\ln\ln n)$. Subsequently, Litvak and Tikhomirov showed that the conjecture also holds for $\xi = \on{Ber}(p_n)$ for $Cn^{-1}\ln n \leq p_n \leq c$, where $c > 0$ is a small absolute constant and $C > 0$ is a large absolute constant. Recent work of Huang \cite{Hua20} was able to bridge the gap between the regimes covered in \cite{BR18} and \cite{Hua20}, leaving open the regime $p \in (c, 1/2)$. Establishing \cref{conj:exact} (as opposed to the stronger \cref{thm:main-exact}) in this regime does not require the full strength of the ideas in this paper -- since this is a case of substantial interest (see Problem 8.2~in the recent work of Litvak and Tikhomirov \cite{LT20}), we have provided the much simpler details in the self-contained companion note \cite{JSS20discrete1}, which also serves as a gentle introduction to the present work and contains complete details of some relatively standard arguments which are only sketched in the present work. 

For general discrete distributions $\xi$, the only previous systematic study in the literature is the work of Bourgain, Vu, and Wood \cite{BVW10}. They show \cite[Corollary~1.2]{BVW10} that if $\xi$ is a discrete distribution with $\sup_{r \in \mb{R}}\on{P}[\xi = r] =: p$, then $q_n(\xi) \le (\sqrt{p} + o_n(1))^{n}$, which is far from optimal (the true bound is never more than $(p+o_n(1))^{n}$, although it may be much smaller). On the other hand, up to a possible $o_n(1)$ term, \cref{thm:main-approximate} in this work always obtains the correct base of the exponent. 

For certain specific distributions, Bourgain, Vu, and Wood obtain the correct base of the exponent (again, up to a $o_n(1)$ term). Specifically, they show \cite[Corollaries~3.1, 3.2]{BVW10} that if $\xi_{1,\mu}$ is a random variable taking on the value $0$ with probability $1-\mu$ and $\pm 1$ with probability $\mu/2$ each, and if $\xi_{2,\mu}$ is a random variable taking on the value $0$ with probability $1-\mu$ and $\pm 1, \pm 2$ with probability $\mu/4$ each, then $q_{n}(\xi_{1,\mu}) = (1-\mu + o_n(1))^{n}$ for all $\mu \in (0,1/2)$ and $q_{n}(\xi_{2,\mu}) = (1-\mu + o_n(1))^{n}$ for all $\mu \in (0,16/25)$. For these random variables, \cref{thm:main-approximate} determines the correct base of the singularity probability for all fixed $\mu \in (0,1)$, and \cref{thm:main-exact} determines the leading order in the asymptotic expansion for $\mu \in (0,1), \mu \neq 2/3$ in the first case, and $\mu \in (0,1), \mu \neq 4/5$ in the second case.
In fact, for $\mu\in(0,2/3)$ in the first case, and $\mu\in(0,4/5)$ in the second case, \cref{thm:main-exact} determines the first two terms in the asymptotic expansion.

We remark that the results of \cite{BVW10} such as \cite[Corollary~1.2]{BVW10} are also applicable to discrete random variables valued in the complex numbers, and settings where the entries of $M_{n}(\xi)$ are not identically distributed, and a small number of rows of $M_{n}(\xi)$ are possibly deterministic; we have not pursued these extensions. 

Finally, we remark that there was a recent paper of Irmatov \cite{Irm20} which claimed to resolve \cref{conj:exact} for Rademacher random matrices. Experts have informed us that there are some unresolved issues in that work that its author is aware of, including \cite[Theorem~3]{Irm20}. Furthermore, upon slight modification, the proof in \cite{Irm20} would appear to give impossibly good error terms.

\subsection{Additional results}\label{sub:miscellaneous}
The next result addresses the main question left open by our work, namely, the resolution of \cref{conj:exact} for discrete distributions $\xi$ which are uniform on their support. In this direction, we provide a sharp analysis of the contribution of a certain low-entropy part of the unit sphere; in fact, it is this contribution which forms the leading term of the conjectured asymptotic expansion of the singularity probability. This theorem is also central to the proofs of \cref{thm:main-exact,thm:main-approximate}. 
\begin{theorem}\label{thm:compressible}
Fix a discrete distribution $\xi$. There exist $\delta,\rho,\eta> 0$ depending on $\xi$ such that for all sufficiently large $n$ and $t\le 1$, 
\[\mb{P}\bigg[\inf_{x\in \on{Cons}(\delta,\rho)} \snorm{M_n(\xi)x}_2\le t\bigg]\le n\mb{P}[\mc{E}_{e_1}] + \binom{n}{2}(\mb{P}[\mc{E}_{e_1-e_2}] + \mb{P}[\mc{E}_{e_1+e_2}]) + (t+\mb{P}[\mc{E}_{e_1 - e_2}])e^{-\eta n}.\]
\end{theorem}
The set $\on{Cons}(\delta, \rho)$ appearing above is the set of unit vectors which have at least $(1-\delta)n$ coordinates within distance $\rho/\sqrt{n}$ of each other (see \cref{def:almost-constant}), although a trivial modification shows this result holds for any sufficiently low-entropy subset of the unit sphere.

\begin{remark}
In the companion note \cite{JSS20discrete1}, as a natural consequence of our study of anti-concentration on slices, we are able to resolve a conjecture of Nguyen \cite[Conjecture~1.4]{Ngu13} on the probability of singularity for a certain combinatorial model of random matrices. Namely, let  $Q_{n}$ be an $n\times n$ random matrix with independent rows, each of which is chosen uniformly from among those vectors in $\{0,1\}^{n}$ which have sum exactly $\lfloor n/2 \rfloor$. In \cite{Ngu13}, Nguyen showed that $\mb{P}[Q_n \text{ is singular}] = O_{C}(n^{-C})$ for any $C > 0$, and conjectured   \cite[Conjecture~1.4]{Ngu13} that $\mb{P}[Q_n \text{ is singular}] = (1/2 + o_n(1))^{n}$. After intermediate work \cite{FJLS19, Jai19}, an exponential upper bound on the singularity probability was only very recently obtained in work of Tran \cite{Tra20}, the key difficulty being establishing even relatively coarse inverse Littlewood--Offord type results for the slice. In \cite{JSS20discrete1}, we settle \cite[Conjecture~1.4]{Ngu13} using the techniques developed in this work.

\begin{theorem}[{\cite[Theorem~1.4]{JSS20discrete1}}]
\label{thm:row-regular}
For every $\epsilon>0$, there exists $C_{\epsilon}$ depending on $\epsilon$ such that for all sufficiently large $n$, and for all $t\ge 0$,  
\[\mb{P}[s_n(Q_n)\le t/\sqrt{n}]\le C_\epsilon t + (1/2+\epsilon)^n.\]
\end{theorem}
\end{remark}

\subsection{Overview of the techniques}
As in many works in this area, we use the high-level strategy (going back to Ka\v{s}in \cite{Kas77} and subsequently used in \cite{Sch04,LPRT05,RV08, Rud08}) of dividing the unit sphere into `structured' and `unstructured' components, and estimating the contribution of each part separately. However, compared to previous works, the treatment of both components require overcoming significant obstacles which unavoidably arise in the sharp analysis of the invertibility of random matrices in any amount of generality. 

For instance, in the analysis of structured vectors, we need to additionally capture the event that two rows/columns of the matrix are equal (up to a sign) whereas previous considerations of sharp invertibility only addressed scenarios where the dominant contribution to the probability of singularity is due to a single row or column being zero. As discussed in the remark after \cref{cor:exact}, this is a fundamental issue. Moreover, in the analysis of unstructured vectors, we need precise metric entropy estimates for the anti-concentration problem with respect to random vectors on general multi-slices. Obtaining partial estimates of this nature (which are not sufficient to prove \cref{conj:exact}) even for the special case of the Boolean slice is already a highly non-trivial endeavor which is at the heart of the recent work of Litvak and Tikhomirov \cite{LT20}, where it is accomplished using the substantially more involved notion of the `UDLCD'. 

\textbf{Structured vectors}: The structured vectors in our work are `almost-constant vectors' i.e.~those vectors on $\mb{S}^{n-1}$ which have $(1-\delta)n$ coordinates within distance $\rho/\sqrt{n}$ of each other, where $\delta, \rho > 0$ are sufficiently small constants. This class of structured vectors arises naturally in the consideration of the anti-concentration property of a sequence of numbers with respect to a random vector constrained to lie in a `slice'. Moreover, since vectors which are close to the standard basis vectors $e_i$ or to $e_i \pm e_j$ clearly play a special role in the problem under consideration, it is natural to separately handle `elementary' and `non-elementary' structured vectors. 

Our treatment of structured vectors, culminating in \cref{thm:compressible}, requires significant innovations compared to previous works on the sharp invertibility of sparse random Bernoulli matrices \cite{BR18, LT20, Hua20} -- in the sparse Bernoulli case, the corresponding class of elementary vectors only needs to consist of those vectors which are close to some $e_i$, and the largest atom of the the random variable $\on{Ber}(p)$ is conveniently at $0$. We refer the reader to \cite{JSS20discrete1} for the much simpler argument in this case. 

In the present work, in order to handle non-elementary vectors, we need to develop novel sharp anticoncentration estimates \cref{prop:p-characterization,prop:p-characterization-II} (in contrast, the essentially standard estimate \cite[Lemma~2.6]{JSS20discrete1} was sufficient for the corresponding step in \cite{JSS20discrete1}). Even more involved is the analysis of elementary vectors, for which we develop a new technique. Let us begin by discussing this technique for $\xi = \on{Ber}(p)$ for fixed $p \in (0,1/2)$, in which case, the elementary vectors are those which are close to some standard basis vector. For concreteness, consider vectors which are sufficiently close to $e_1$. We show that, if any such vector has exponentially small
image, then either the first column of the matrix is the zero vector, or it must belong to
a universal subset of nonzero vectors of $\{0,1\}^n$ of measure at most $(1 - p + \epsilon)^n$. 
The first case corresponds to the term $\mb{P}[\mc{E}_{e_1}]$ in \cref{conj:exact}; for the second case, we leverage the seminal work of Rudelson
and Vershynin to show that, on our event, the probability that any vector in this universal subset
appears as the first column of the matrix is at most $\exp(-4\epsilon n)$, at which point we can conclude using
the union bound. 

Of course, for general discrete random variables $\xi$, one must enlarge the class of elementary vectors
to include unit vectors which are close to $(e_i \pm e_j)/\sqrt{2}$ and unit vectors which are close to $e_i$. In the first case (\cref{prop:e1-e2,prop:e1+e2}), we use a rotation trick to reduce to a situation where we can use an analysis similar to (but more complicated than) the one outlined in the previous paragraph. The second case requires a very careful treatment since we are aiming for a leading term of the form $(\mb{P}[\xi = 0])^{n}$ (as opposed to $(\sup_{r \in \mb{R}}\mb{P}[\xi = r])^{n}$), and moreover, the desired error is $(\mb{P}[\xi = \xi'] - \eta)^{n}$ which may be very small. To accomplish this, we first prove a version of \cref{thm:main-approximate} with an estimate on the singularity probability of the form  $(\sup_{r \in \mb{R}}\mb{P}[\xi = r] + o_n(1))^{n}$ (\cref{prop:initial-unstructured,thm:weak}), and then leverage these preliminary estimates to obtain the desired bound. 

We emphasize that our treatment of structured vectors, as captured by \cref{thm:compressible}, is not sensitive to the non-uniformity of the distribution $\xi$. In particular, given \cref{thm:main-exact,thm:compressible}, the only missing case in the complete resolution of \cref{conj:exact} (in fact, in a stronger form) is a sharp analysis of unstructured vectors in the case when $\xi$ is uniform on its support.

\textbf{Unstructured vectors}: The unstructured vectors are the complement of the structured vectors i.e.~ those which do not have a $(1-\delta)$-fraction of their coordinates within $\rho/\sqrt{n}$ of each other. Our treatment of these vectors relies on the non-uniformity of $\xi$ by exploiting the gap between $\mb{P}[\xi = \xi']$ and the entropy of $\xi$; the idea to exploit such a gap to prove sharp invertibility results (in the case of Bernoulli random variables) is due to Litvak and Tikhomirov \cite{LT20}. 

The main ingredient in our work for handling such vectors is \cref{thm:inversion-of-randomness}, which is an extension of \cite[Theorem~B]{Tik20} to a (real) multislice, i.e., the set of vectors in $\{a_1,\dots,a_{k}\}^{n}$ which have a prescribed number of coordinates taking on each of the values $a_1,\dots,a_k$. Such a result was previously not known even for the Boolean slice; indeed, the work \cite{LT20} uses a rather involved notion of arithmetic structure to study anti-concentration on Boolean slices, which is not powerful enough to handle slices that are \emph{not} very far from the central slice. We remark that in general, even establishing much less precise versions of \cite[Theorem~B]{Tik20} on the Boolean slice has been very challenging, despite much work due to the natural connection to certain combinatorial models of random matrices (cf.~\cite{JSS20digraphs} and the references therein).

Compared to \cite[Theorem~B]{Tik20}, we need to overcome two challenges. The first, as mentioned above, is the lack of independence between the coordinates of a vector uniformly distributed on the multi-slice. The second challenge is that $a_1,\dots,a_k$ are now arbitrary real numbers (corresponding to the support of $\xi$), and hence, certain arguments tailored for integers no longer apply. Overcoming these challenges requires additional ideas, which we discuss in \cref{sec:multislice}. The reader is also referred to \cite[Theorem~3.1]{JSS20discrete1} where we record the proof for the Boolean slice, a setting which encounters the first challenge but not the second. 

\subsection{Notation}
For a positive integer $N$, $\mb{S}^{N-1}$ denotes the set of unit vectors in $\mb{R}^{N}$, and if $x\in\mb{R}^N$ and $r\ge 0$ then $\mb{B}_2^N(x,r)$ denotes the radius $r$ Euclidean ball in $\mb{R}^{N}$ centered at $x$. $\snorm{\cdot}_2$ denotes the standard Euclidean norm of a vector, and for a matrix $A = (a_{ij})$, $\snorm{A}$ is its spectral norm (i.e., $\ell^{2} \to \ell^{2}$ operator norm).

We will let $[N]$ denote the interval $\{1,\dots, N\}$. For nonnegative integers $m\le n$, we let $\{0,1\}^n_m$ be the set of vectors in $\{0,1\}^{n}$ with sum $m$.

Since it is essential throughout the paper, we formally record the definition of a discrete random variable and the corresponding random matrix. 

\begin{definition}
We say that a random variable $\xi$ is a discrete random variable (equivalently, has a discrete distribution) if it is real-valued, its support is finite, and the support contains at least two distinct points. $M_{n}(\xi)$ denotes the $n\times n$ random matrix, with independent entries that are copies of $\xi$.  
\end{definition}

For $\xi$ a discrete random variable with $k = |\on{supp}(\xi)|$ (so that $k\ge 2$), we will denote its support by $\vec{a} = (a_1,\dots,a_k)$, and the (nonzero) probabilities of attaining $a_1,\dots,a_k$ by $\vec{p} = (p_1,\dots,p_k)$. Note that $\snorm{\vec{p}}_{1} = 1$, and $\snorm{\vec{p}}_{2}^2\le\snorm{\vec{p}}_{\infty}$ with equality if and only if $\xi$ is uniform on its support. We will use $H(\xi)$ to denote the natural-logarithmic entropy of $\xi$, i.e., $H(\xi) = H(\vec{p}) = \sum_{i=1}^{k} -p_i\log(p_i)$. We will (somewhat abusively) use $p_0$ to denote $\mb{P}[\xi = 0]$. 

For a random variable $\xi$ and a real number $r \ge 0$, we let $\mc{L}(\xi, r):= \sup_{z \in \mb{R}}\mb{P}[|\xi-z|\le r]$. We will use $\ell_1(\mb{Z})$ to denote the set of functions $f\colon \mb{Z} \to \mb{R}$ for which $\sum_{z \in \mb{Z}}|f(z)| < \infty$.

We will also make use of asymptotic notation. For functions $f,g$, $f = O_{\alpha}(g)$ (or $f\lesssim_{\alpha} g$) means that $f \le C_\alpha g$, where $C_\alpha$ is some constant depending on $\alpha$; $f = \Omega_{\alpha}(g)$ (or $f \gtrsim_{\alpha} g$) means that $f \ge c_{\alpha} g$, where $c_\alpha > 0$ is some constant depending on $\alpha$, and $f = \Theta_{\alpha}(g)$ means that both $f = O_{\alpha}(g)$ and $f = \Omega_{\alpha}(g)$ hold. For parameters $\epsilon, \delta$, we write $\epsilon \ll \delta$ to mean that $\epsilon \le c(\delta)$ for a sufficient function $c$. 

Finally, we will omit floors and ceilings where they make no essential difference.

\subsection{Organization}
The remainder of this paper is organized as follows. In \cref{sec:multislice}, we prove our key inversion of randomness estimate for conditional thresholds on the multislice, \cref{thm:inversion-of-randomness}. In \cref{sec:incompressible}, we use this to prove the invertibility estimate for unstructured vectors (\cref{thm:incompressible}). In the short \cref{sec:initial-compressible}, we prove a weaker version of \cref{thm:main-approximate}; this is used in our treatment of structured vectors (i.e., the proof of \cref{thm:compressible}), which is the content of \cref{sec:compressible}. In \cref{sec:deduction-main}, we quickly combine \cref{thm:compressible,thm:incompressible} to prove \cref{thm:main-exact,thm:main-approximate}.

\subsection{Acknowledgements}
We thank Mark Rudelson, Konstantin Tikhomirov, and Yufei Zhao for comments on the manuscript. A.S.~and M.S.~were supported by the National Science Foundation Graduate Research Fellowship under Grant No.~1745302. This work was done when V.J.~was participating in a program at the Simons Institute for the Theory of Computing.

\section{Inversion of randomness on the multislice}\label{sec:multislice}
In this section, we prove our key inversion of randomness result, \cref{thm:inversion-of-randomness}. We will focus on the non-independent ``multislice'' version as its deduction is strictly harder than the independent version, \cref{cor:independent-threshold-inversion} (which we will only use to establish the preliminary estimate \cref{thm:weak}). 

The proof of \cref{thm:inversion-of-randomness} follows a direction introduced by
Tikhomirov \cite{Tik20}. In this approach, the relevant L\'evy concentration function of a random vector is
replaced with certain random averages of functions. One then shows that the random vectors with
large values of the L\'evy concentration function are super-exponentially rare, by first demonstrating
a weaker notion of anticoncentration after revealing $(1-\epsilon)n$ coordinates of the random vector,
and then iterating a smoothing procedure on linear-sized pieces of the vector which allows one
to bootstrap the strength of anticoncentration considered.

Our major challenges lie in (i) the non-independence of the coordinates of a vector on the multislice, as the arguments in \cite{Tik20} rely strongly on the independence structure of the considered model, and (ii) the freedom to allow the support of $\xi$ to consist of arbitrary real numbers, as certain arguments in \cite{Tik20} rely on the integrality of the support. 

The reader may benefit from first perusing a simpler version of this argument, for slices of the Boolean hypercube, which is presented in the companion note \cite[Section~3]{JSS20discrete1}. We will often note the corresponding statements in \cite{JSS20discrete1} for the reader's convenience. 

\subsection{Statement and preliminaries}\label{sub:smoothing-statement}
Let $N,n\ge 1$ be integers and let $0 < \delta < 1/4$, $K_3 > K_2 > K_1 > 1$ be real parameters. We say that $\mc{A}\subseteq\mb{Z}^n$ is $(N,n,K_1,K_2,K_3,\delta)$-admissible if
\begin{itemize}
    \item $\mc{A} = A_1\times\cdots\times A_n$, where each $A_i$ is a subset of $\mb{Z}$,
    \item $|A_1|\cdots|A_n|\le (K_3N)^n$,
    \item $\max_i\max\{|a|: a\in A_i\}\le nN$,
    \item $A_i$ is an integer interval of size at least $2N+1$ for $i > 2\delta n$,
     and either (P1) and (P2) hold, or (Q1) and (Q2) hold:
\end{itemize}
\begin{enumerate}[(P1)]
    \item $A_{2i}$ is an integer interval of size at least $2N+1$ contained in $[-K_1N,K_1N]$ for $i\le\delta n$,
    \item $A_{2i-1}$ is symmetric about $0$, is a union of two integer intervals of total size at least $2N$, and satisfies $A_{2i-1}\cap[-K_2N,K_2N] = \emptyset$ for $i\le\delta n$.
\end{enumerate}
\begin{enumerate}[(Q1)]
    \item $A_{2i}$ is an integer interval of size at least $2N+1$ contained in $[K_1N,K_2N]$ for $i\le\delta n$,
    \item $A_{2i-1}$ is an integer interval of size at least $2N+1$ contained in $[-K_2N,-K_1N]$ for $i\le\delta n$.
\end{enumerate}

Recall at this point that $\xi$, which has (nonzero) probabilities $\vec{p} = (p_1,\ldots,p_k)$ on atoms $\vec{a} = (a_1,\ldots,a_k)$, is fixed. Let $\mc{A} = A_1\times\cdots\times A_n$ be an $(N,n,K_1,K_2,K_3,\delta)$-admissible set, and let $(X_1,\ldots,X_n)$ be the random vector uniformly distributed on $\mc{A}$. For any $f: \mb{R}\to\mb{R}$, any $0\le\ell\le n$, and any $\vec{s}\in\mb{Z}_{\ge 0}^k$ with $\snorm{\vec{s}}_1 = \ell$, define the random function (depending on the randomness of $X_1,\dots,X_n$):
\[f_{\mc{A},\vec{s},\ell}(t) := \mb{E}_b\bigg[f\bigg(t+\sum_{i=1}^\ell b_iX_i\bigg)\bigg|\#\{b_i=a_j\}=s_j \forall j \in [k]\bigg],\]
where $\mb{E}_b$ denotes the expectation over a random vector $b=(b_1,\dots,b_\ell)\in\mb{R}^\ell$ with coordinates independently distributed as $\xi$. The conditioning encodes that for all $j \in [k]$, there are exactly $s_j$ coordinates (out of $\ell$) where $b$ hits the atom $a_j$.

\begin{theorem}\label{thm:inversion-of-randomness}
Fix a discrete distribution $\xi$. For $0 < \delta < 1/4$, $K_3 > K_2 > K_1 > 1$, $\epsilon\ll\min(\vec{p})$, and a given parameter $M\ge 1$, there are  $L_{\ref{thm:inversion-of-randomness}} = L_{\ref{thm:inversion-of-randomness}}(\xi,\epsilon,\delta,K_1,K_2,K_3) > 0$, and $\gamma_{\ref{thm:inversion-of-randomness}} = \gamma_{\ref{thm:inversion-of-randomness}}(\xi,\epsilon,\delta,K_1,K_2,K_3)\in(0,\epsilon)$ independent of $M$ and $n_{\ref{thm:inversion-of-randomness}} = n_{\ref{thm:inversion-of-randomness}}(\xi,\epsilon,\delta,K_1,K_2,K_3,M)\ge 1$ and $\eta_{\ref{thm:inversion-of-randomness}} = \eta_{\ref{thm:inversion-of-randomness}}(\xi,\epsilon,\delta,K_1,K_2,K_3,M)$ such that the following holds.

Let $n\ge n_{\ref{thm:inversion-of-randomness}}$, $1\le N\le\exp((H(\vec{p})-\epsilon)n)$, $f\in L^1(\mb{R})$ be a nonnegative function such that $\snorm{f}_1 = 1$ and $\log_2 f$ is $\eta_{\ref{thm:inversion-of-randomness}}$-Lipschitz, and $\mc{A}$ be $(N,n,K_1,K_2,K_3,\delta)$-admissible. Suppose also that $\snorm{\vec{\gamma}}_\infty\le\gamma_{\ref{thm:inversion-of-randomness}}$. Then, for any $\vec{m} \in \mb{Z}_{\ge 0}^{k}$ such that $\snorm{\vec{m}}_{1} = n$ and $\snorm{\vec{m}-\vec{p}n}_{\infty}\le\gamma_{\ref{thm:inversion-of-randomness}}n$,
\[\mb{P}[\snorm{f_{\mc{A},\vec{m}, n}}_\infty\ge L(N\sqrt{n})^{-1}]\le \exp(-Mn).\]
\end{theorem}

Given this we can deduce the following corollary which is crucial in our application.

\begin{definition}
\label{def:conditional-levy}
Fix a discrete distribution $\xi$. Let $\vec{\gamma}$ be a nonnegative vector with $\snorm{\vec{\gamma}}_\infty\in(0,\min(\vec{p}))$ and let $r\ge 0$. For a vector $(x_1,\dots,x_n) \in \mb{R}^{n}$, we define
\begin{align*}
    \mc{L}_{\xi,\vec{\gamma}}\bigg(\sum_{i=1}^nb_i x_i, r\bigg) := \sup_{z\in \mb{R}}\mb{P}\bigg[\bigg|\sum_{i=1}^nb_i x_i - z \bigg| \le r \bigg|\#\{b_i=a_j\}\in [p_jn - \gamma_jn, p_jn + \gamma_jn] \forall j \in [k]\bigg],
\end{align*}
where $b_1,\dots,b_n$ are independent $\xi$ random variables. We also define
\[\mc{L}_\xi\bigg(\sum_{i=1}^{n}b_ix_i, r\bigg) = \sup_{z\in\mb{R}}\mb{P}\bigg[\bigg|\sum_{i=1}^nb_ix_i-z\bigg|\le r\bigg].\]

\end{definition}

\begin{corollary}\label{cor:threshold-inversion}
Fix a discrete distribution $\xi$. For $0 < \delta < 1/4$, $K_3 > K_2 > K_1 > 1$, $\epsilon\ll\min(\vec{p})$, and a given parameter $M\ge 1$, there are  $L_{\ref{cor:threshold-inversion}} = L_{\ref{cor:threshold-inversion}}(\xi,\epsilon,\delta,K_1,K_2,K_3) > 0$ and $\gamma_{\ref{cor:threshold-inversion}} = \gamma_{\ref{cor:threshold-inversion}}(\xi,\epsilon,\delta,K_1,K_2,K_3)\in(0,\epsilon)$ independent of $M$ and $n_{\ref{cor:threshold-inversion}} = n_{\ref{cor:threshold-inversion}}(\xi,\epsilon,\delta,K_1,K_2,K_3,M)\ge 1$ such that the following holds.

Let $n\ge n_{\ref{cor:threshold-inversion}}$, $1\le N\le\exp((H(\vec{p})-\epsilon)n)$ and $\mc{A}$ be $(N,n,K_1,K_2,K_3,\delta)$-admissible. Suppose also that $\snorm{\vec{\gamma}}_\infty\le\gamma_{\ref{cor:threshold-inversion}}$. Then
\[\bigg| \bigg\{x\in \mc{A}: \mc{L}_{\xi,\vec{\gamma}}\bigg(\sum_{i=1}^nb_ix_i,\sqrt{n}\bigg)\ge L_{\ref{cor:threshold-inversion}}N^{-1}\bigg\}\bigg|\le e^{-Mn}|\mc{A}|.\]
\end{corollary}
\begin{proof}[Proof sketch]
This is essentially the same as the deduction in \cite[Corollary~4.3]{Tik20}. We apply \cref{thm:inversion-of-randomness} to $f(t):= 2^{-|t|/\sqrt{n}}/\iota$, where $t \in \mb{R}$ and $\iota$ is an appropriate normalization, separately for all $\vec{m} \in \mb{Z}_{\ge 0}^{k}$ such that $\snorm{\vec{m}-\vec{p}n}_\infty\le\gamma_{\ref{thm:inversion-of-randomness}}n$,
and then conclude using a union bound.
\end{proof}

The proof of \cref{thm:inversion-of-randomness} makes use of an anticoncentration estimate on the multislice, which we record below (\cref{lem:mslkr,lem:slice-levy-kolmogorov-rogozin}), 
and is ultimately a consequence of the following standard anticoncentration inequality due to Kolmogorov-L\'evy-Rogozin.

\begin{lemma}[\cite{Rog61}]
\label{lem:LKR}
Let $\xi_1,\dots, \xi_{n}$ be independent random variables. Then, for any real numbers $r_1,\dots,r_n > 0$ and any real $r \ge \max_{i \in [n]}r_i$, we have
\begin{align*}
    \mc{L}\bigg(\sum_{i=1}^{n}\xi_i, r\bigg) \le \frac{C_{\ref{lem:LKR}}r}{\sqrt{\sum_{i=1}^{n}(1-\mc{L}(\xi_i, r_i))r_i^{2}}},
\end{align*}
where $C_{\ref{lem:LKR}}>0$ is an absolute constant. 
\end{lemma}

\begin{lemma}\label{lem:mslkr}
Fix $(a_1,\ldots,a_k)\in\mb{R}^k$ with distinct coordinates. Let $\sigma, \lambda \in (0,1/3)$ and $r > 0$. Let $Z = \{z_1,\ldots,z_n\}$ be a set of real numbers for which there exist disjoint subsets $Z_1, Z_2 \subseteq Z$ such that $|Z_1|, |Z_2| \ge \sigma n$ and such that $|z_i - z_j| \ge r$ for all $z_i \in Z_1, z_j \in Z_2$. Then, there exists $C_{\ref{lem:mslkr}} = C_{\ref{lem:mslkr}}(\lambda, \sigma, k)$ such that for any $\vec{s}\in\mb{Z}_{\ge 0}^k$ with $\snorm{\vec{s}}_1 = n$ and with $s_\ell\in[\lambda n,(1-\lambda)n]$ for some $\ell \in [k]$, we have
\[\mc{L}\bigg(\sum_{i=1}^{n}z_ib_i, r\cdot \min_{i<j}|a_i-a_j|\bigg) \le\frac{C_{\ref{lem:mslkr}}}{\sqrt{n}},\]
where $(b_1,\dots, b_n)$ is a random vector uniformly chosen from among those with $s_j$ coordinates equal to $a_j$ for all $j \in [k]$.
\end{lemma}
\begin{proof}
By reindexing the coordinates of $Z$, we may assume that for $i \in [\sigma n]$, $z_{2i - 1} \in Z_1$ and $z_{2i} \in Z_2$. In particular, for $i \in [\sigma n]$, we have $|z_{2i} - z_{2i - 1}| \ge r$. Furthermore, by the pigeonhole principle, there exists some $\ell' \neq \ell$ such that $s_{\ell'}\ge \lambda n/k$. We will now use the randomness within the atoms $a_\ell$ and $a_{\ell'}$ in order to derive the anticoncentration result. Note that $\sum_{i=1}^{n} b_iz_i$
has the same distribution as
\[\sum_{i>2\sigma n} z_ib_i + \sum_{j\le \sigma n}\bigg(z_{2j-1}b_{2j-1} + z_{2j}b_{2j} + b_j'(b_{2j} - b_{2j-1})(z_{2j-1} - z_{2j})\bigg),\]
where $b_1',\dots,b_{\sigma n}'$ are i.i.d.~$\on{Ber}(1/2)$ random variables. Next, note that by a standard large deviation estimate, we have
\begin{align}
\label{eq:rerandomize}
\mb{P}[|\{j \in [\sigma n]: \{b_{2j-1}, b_{2j}\} = \{a_\ell,a_{\ell'}\}| \le c(\sigma, \lambda, k)n] \le \exp(-c(\sigma, \lambda, k)n),
\end{align}
where $c(\sigma, \lambda, k) > 0$ is a constant depending only on $\sigma$, $\lambda$, and $k$. On the other hand, on the complement of this event, we may conclude by applying \cref{lem:LKR} to \cref{eq:rerandomize}, using only the randomness in $b_1',\dots, b_{\sigma n}'$.
\end{proof}

\begin{lemma}\label{lem:slice-levy-kolmogorov-rogozin}
Fix a discrete distribution $\xi$, $\lambda\in(0,1/3)$,  $\delta_0\in(0,1/4)$. Let $\mc{A}$ be $(N,n,K_1,K_2,K_3,\delta)$-admissible for some integer parameters $N,n$ and real parameters $\delta\in[\delta_0,1/4)$, $K_3 > K_2 > K_1 > 1$. Suppose that $n > n_{\ref{lem:slice-levy-kolmogorov-rogozin}}(\lambda,\delta_0,K_1,K_2,K_3)$, $\ell \ge \delta_0n$, and $\vec{s} \in \mb{Z}_{\ge 0}^{k}$ with $\snorm{s}_{1} = \ell$ and $s_{j_0} \in  [\lambda \ell, (1-\lambda)\ell]$ for some $j_0 \in [k]$. Then, for any interval $J$,
\[\int_{t\in J}f_{\mc{A},\vec{s},\ell}(t)dt\le\frac{C_{\ref{lem:slice-levy-kolmogorov-rogozin}}(\lambda,\xi, \delta_0,K_1,K_2)\max(|J|,N)}{N\sqrt{n}}.\]
\end{lemma}
\begin{proof}
The proof is nearly identical to that in \cite[Lemma~4.4]{Tik20} though we provide details as we are in the slightly different setting of $L^1(\mb{R})$. Fix $X_1,\ldots,X_\ell$. Then
\begin{align*}
\int_{t\in J}f_{\mc{A},\vec{s},\ell}(t)dt &= \int_{t\in J} \mb{E}_b\bigg[f\bigg(t+\sum_{i=1}^\ell b_iX_i\bigg)\bigg|\#\{b_i=a_j\}=s_j \forall j \in [k]\bigg] dt\\
&= \mb{E}_b\bigg[\int_{t\in J}  f\bigg(t+\sum_{i=1}^\ell b_iX_i\bigg) dt\bigg|\#\{b_i=a_j\}=s_j \forall j \in [k]\bigg] \\
&= \mb{E}_b\bigg[\int_{t\in\mb{R}}f(t)\mbm{1}_{ J+\sum_{i=1}^\ell b_iX_i}(t)dt\bigg|\#\{b_i=a_j\}=s_j \forall j \in [k]\bigg]\\
&= \int_{t\in\mb{R}}f(t)\mb{E}_b\bigg[\mbm{1}_{ J+\sum_{i=1}^\ell b_iX_i}(t)\bigg|\#\{b_i=a_j\}=s_j\forall j \in [k]\bigg]dt\\
&= \int_{t\in\mb{R}}f(t)\mb{P}_b\bigg[\sum_{i=1}^\ell b_iX_i\in J-t\bigg|\#\{b_i=a_j\}=s_j\forall j \in [k]\bigg]dt\\
&\le \mc{L}\bigg(\sum_{i=1}^{\ell}b_i X_i, |J|\bigg)\int_{t\in \mb{R}}|f(t)|dt \le \mc{L}\bigg(\sum_{i=1}^{\ell}b_i X_i, |J|\bigg),
\end{align*}
where $(b_1,\dots,b_\ell)$ is uniformly chosen from vectors which have $s_j$ coordinates equal to $a_j$ for all $j \in [k]$, and we have used that $\snorm{f}_{1} = 1$. 
The required estimate now follows immediately from \cref{lem:mslkr} applied with $r = (K_2-K_1)N$, which is possible due to the admissibility of $\mc{A}$. 
\end{proof}

\subsection{Preprocessing on real-valued multislices}\label{sub:preprocessing}
As in \cite{Tik20}, we first prove a version of \cref{thm:inversion-of-randomness} in which $L$ is allowed to depend on $M$. 
\begin{proposition}\label{prop:rough-Linfty}
Fix a discrete distribution $\xi$. For $0 < \delta < 1/4$, $K_3 > K_2 > K_1 > 1$, $\epsilon\ll\min(\vec{p})$, and a given parameter $M\ge 1$, there is $\gamma_{\ref{prop:rough-Linfty}} = \gamma_{\ref{prop:rough-Linfty}}(\xi,\epsilon,\delta,K_1,K_2,K_3)\in(0,\epsilon)$ independent of $M$ and there are $L_{\ref{prop:rough-Linfty}} = L_{\ref{prop:rough-Linfty}}(\xi,\epsilon,\delta,K_1,K_2,K_3,M) > 0$ and $n_{\ref{prop:rough-Linfty}} = n_{\ref{prop:rough-Linfty}}(\xi,\epsilon,\delta,K_1,K_2,K_3,M)\ge 1$ such that the following holds. 

Let $n\ge n_{\ref{prop:rough-Linfty}}$, $1\le N\le\exp((H(\vec{p})-\epsilon)n)$, and $\mc{A}$ be $(N,n,K_1,K_2,K_3,\delta)$-admissible. Let $f$ be a nonnegative function in $L^1(\mb{R})$ with $\snorm{f}_1 = 1$ such that $\log_2{f}$ is $1$-Lipschitz. Then, for all $\ell\in[(1-\gamma_{\ref{prop:rough-Linfty}})n,n]$ and $\vec{s}\in \mb{Z}_{\ge 0}^{k}$ with $\snorm{\vec{s}}_{1} = \ell$ and $\snorm{\vec{s}-\vec{p}\ell}_\infty\le\gamma_{\ref{prop:rough-Linfty}}\ell$, we have
\[\mb{P}\bigg[\snorm{f_{\mc{A},\vec{s},\ell}}_\infty\ge L_{\ref{prop:rough-Linfty}}(N\sqrt{n})^{-1}\bigg]\le\exp(-Mn).\]
\end{proposition}

For a simpler version of this argument, we refer the reader to the proof of \cite[Proposition~3.6]{JSS20discrete1}. \cref{prop:rough-Linfty} should be seen as an analogue of \cite[Lemma~4.6]{Tik20} for the multislice. As mentioned earlier, compared to \cite{Tik20}, our situation is much more delicate since we are working with a vector with non-independent coordinates and need to extract a term corresponding to the entropy of the multislice. Such complications are already encountered when working with a Boolean slice, as in \cite[Proposition~3.6]{JSS20discrete1}. Working on real multislices presents additional difficulties (along with significant notational complications), owing to the fact that we are working on $L^1(\mb{R})$; this extension is handled by using the log-Lipschitz condition on $f$. We note that the corresponding statements in \cite{JSS20discrete1, Tik20} do not need to use any log-Lipschitz assumption at this stage of the argument since they are proved for $\ell^{1}(\mb{Z})$. We also note that, while the constant $1$ in $1$-log-Lipschitz is arbitrary, some condition of this nature is necessary to rule out $f$ being very close to a Dirac mass (\cite{Tik20}).

We first note the trivial recursive relation
\[f_{\mc{A},\vec{s},\ell}(t) = \sum_{i=1}^k\frac{s_i}{\ell}f_{\mc{A},\vec{s}-e_i,\ell-1}(t+a_iX_\ell)\]
for all $1\le\ell\le n$ and $\vec{s}\in\mb{Z}_{\ge 0}^k$ with $\snorm{\vec{s}}_{1} = \ell$. If any coordinate of $\vec{s}$ is zero, note the corresponding term (which would be undefined) has a coefficient of $0$, and drops out. Note also that, by definition, $f_{\mc{A},\vec{0},0} = f$.

\begin{definition}[Step record and averaging sequence]
\label{def:record-averaging}
Fix $f, \mc{A}, \vec{s}, \ell$, a point $t \in \mb{R}$, and a choice of $X = (X_1,\dots,X_n)$. For such a choice, we define the \emph{averaging sequence} $(t_i)_{i=0}^{\ell}$ and \emph{step record} $(w_i)_{i=1}^{\ell}$ as follows:
\begin{itemize}
    \item $t_\ell := t$,
    \item Since \[h_\ell := f_{\mc{A},\vec{s},\ell}(t_\ell) = \sum_{j=1}^k\frac{s_j}{\ell}f_{\mc{A},\vec{s}-e_j,\ell-1}(t_\ell+a_jX_\ell),\]
    at least one of the $k$ terms $f_{\mc{A},\vec{s}-e_j,\ell-1}(t_\ell+a_jX_\ell)$ has a positive coefficient and is at least $h_\ell$. If it is index $j$, set $w_\ell = j$.
    \item Set $t_{\ell-1} := t_\ell + a_{w_\ell}X_\ell$, $h_{\ell-1} := f_{\mc{A}, \vec{s}-e_{w_\ell}, \ell-1}(t_{\ell-1})$, and repeat with $t_{\ell-1},\vec{s}-e_{w_\ell},\ell-1$.
\end{itemize}
It will be convenient to write
\begin{itemize}
    \item $W_i(j):= \#\{u\in[i]: w_u = j\}$ and $\ol{W}_i(j) := W_i(j)/i$ for all $i \in [\ell]$ and $j\in[k]$. We will view $W_i = (W_i(1),\dots, W_i(k))$ as a vector in $\mb{Z}^{k}$.
\end{itemize}
\end{definition}

We note some straightforward consequences of these definitions. 
\begin{itemize}
\item $W_\ell = \vec{s}$.
\item $W_{i-1} = W_{i}-e_{w_i}$ for $1\le i\le\ell$, where we assume $W_0 = \vec{0}$.
\item $\snorm{W_{i}}_{1} = i$.
\item $t_{i-1} = t_i + a_{w_i} X_i$ for all $i \in [\ell]$.
\item $f_{\mc{A}, W_i, i}(t_i) = \sum_{j=1}^k\ol{W_i}(j)f_{\mc{A},W_i-e_j,i-1}(t_i+a_jX_i)$.
\item $h_i = f_{\mc{A}, W_i, i}(t_i)$.
\item $f(t_0) = h_0 \ge h_1 \ge \cdots \ge h_{\ell} = f_{\mc{A},\vec{s},\ell}(t)$.
\end{itemize}
\begin{definition}[Drops and robust steps]
\label{def:drop-robust}
With notation as above, given $i \in [\ell]$:
\begin{itemize}
\item For $\lambda \in (0,1)$, we say that step $i$ is \emph{$\lambda$-robust} if 
\begin{align*}
    \ol{W}_i(w_i) \in (\lambda, 1-\lambda)
\end{align*}
\item For $R > 0$, we say that there is an \emph{$R$-drop} at step $i$ if
\begin{align*}
    f_{\mc{A}, W_i - e_j, i-1}(t_{i-1} + z X_i) \le \frac{R}{N\sqrt{n}}
\end{align*}
for all $j\in[k]$ such that $W_i(j) > 0$ and for all $z\in\on{supp}(\xi-\xi')\setminus\{0\}$.
\end{itemize}
\end{definition}

Next we show that if $\snorm{f_{\mc{A},\vec{s},\ell}}_\infty$ is large in an appropriate sense, then there is a step record and averaging sequence with linearly many robust steps which do not participate in an $R$-drop.

\begin{lemma}\label{lem:entropy}
Let $\xi, \mc{A}, f, N, \epsilon$ be as in \cref{prop:rough-Linfty}, and let $L \ge 1$. Then, there exist $\lambda_{\ref{lem:entropy}} = \lambda_{\ref{lem:entropy}}(\xi,\epsilon) \in (0,1/3)$, $\gamma_{\ref{lem:entropy}} = \gamma_{\ref{lem:entropy}}(\xi,\epsilon) \in (0,1)$, and $n_{\ref{lem:entropy}} = n_{\ref{lem:entropy}}(\xi,\epsilon)$ for which the following holds.

Let $n \ge n_{\ref{lem:entropy}}$, $R = \gamma_{\ref{lem:entropy}}L$, let  $\ell\in[(1-\gamma_{\ref{lem:entropy}})n,n]$ and $\vec{s}\in\mb{Z}_{\ge 0}^k$ satisfy $\snorm{\vec{s}}_1 = \ell$ and $\snorm{\vec{s}-\vec{p}\ell}_\infty\le\gamma_{\ref{lem:entropy}}\ell$. Then, for $(X_1,\dots,X_n) \in \mc{A}$,
\[\snorm{f_{\mc{A},\vec{s},\ell}}_\infty\ge L(N\sqrt{n})^{-1}\]
implies that there exists some $t\in\mb{R}$ with $f_{\mc{A},\vec{s},\ell}(t) \ge L(N\sqrt{n})^{-1}$ so that its averaging sequence $(t_i)_{i=0}^{\ell}$ and step record $(w_i)_{i=1}^{\ell}$ satisfy
\[\#\{i\in[\ell]: \emph{step }i\emph{ is }\lambda_{\ref{lem:entropy}}\emph{-robust and is not an }R\emph{-drop}\}\ge\gamma_{\ref{lem:entropy}} n.\]
\end{lemma}
\begin{proof}
Consider $(X_1,\dots,X_n) \in \mc{A}$ satisfying $\snorm{f_{\mc{A}, \vec{s}, \ell}}_{\infty} \ge L(N\sqrt{n})^{-1}$. Then, there is some $t \in \mb{R}$ such that $f_{\mc{A},s,\ell}(t) \ge L(N\sqrt{n})^{-1}$. We will show that the conclusion of the lemma is satisfied for this $t$, for suitable choice of $\gamma_{\ref{lem:entropy}}, \lambda_{\ref{lem:entropy}}$. Below, we will make extensive use of the notation and relations in \cref{def:drop-robust,def:record-averaging}. Let $(t_i)_{i=0}^{\ell}$ and $(w_i)_{i=1}^{\ell}$ denote, respectively, the averaging sequence and step record of $t$.
Note that
\[L(N\sqrt{n})^{-1}\le f_{\mc{A},\vec{s},\ell}(t) = h_0\prod_{i=1}^\ell\frac{h_i}{h_{i-1}} \le h_{\ell-1} \le \cdots \le h_{0}.\]

We begin by controlling the ratios $h_i/h_{i-1}$ at steps $i$ which are $R$-drops. Hence, suppose that step $i$ is an $R$-drop. If $w_i = u$, then $W_i = W_{i-1}+e_u$ and $t_i=t_{i-1}-a_uX_i$. Hence
\begin{align*}
\frac{h_i}{h_{i-1}} &= \sum_{j=1}^k\ol{W}_i(j)\frac{f_{\mc{A},W_i-e_j,i-1}(t_i+a_jX_i)}{f_{\mc{A},W_{i-1},i-1}(t_{i-1})}\\
&= \ol{W}_i(u) + \sum_{j\neq u}\ol{W}_i(j)\frac{f_{\mc{A},W_i-e_j,i-1}(t_{i-1}+(a_j-a_u)X_i)}{h_{i-1}}\\
&\le\ol{W}_i(u) + \sum_{j\neq u}\ol{W}_i(j)\frac{R(N\sqrt{n})^{-1}}{L(N\sqrt{n})^{-1}}\\
&= \ol{W}_i(u)+(1-\ol{W}_i(u))\gamma_{\ref{lem:entropy}}.
\end{align*}
The inequality uses is the definition of $R$-drops (this is applicable since $a_j-a_u\in\on{supp}(\xi-\xi')\setminus\{0\}$) along with $h_i\ge L(N\sqrt{n})^{-1}$. Note that the condition $W_i(j) > 0$ in the definition of $R$-drops is not satisfied, then the $j$th term already drops out in the first line. Thus, we see that if step $i$ is an $R$-drop, then
\begin{align}
\label{eq:ratio-bound}
    \frac{h_i}{h_{i-1}} \le\ol{W}_i(w_i)+(1-\ol{W}_i(w_i))\gamma_{\ref{lem:entropy}}.
\end{align}

Note that if step $i$ is $\lambda_{\ref{lem:entropy}}$-robust, the right-hand side is at least $\lambda_{\ref{lem:entropy}}$. Therefore, for any step $i$ which is $\lambda_{\ref{lem:entropy}}$-robust, we have
\begin{align}
\label{eq:robust-step-bound}
    \lambda_{\ref{lem:entropy}}
    \le\ol{W}_i(w_i)+(1-\ol{W}_i(w_i))\gamma_{\ref{lem:entropy}}\le\ol{W}_i(w_i)\bigg(1+\frac{\gamma_{\ref{lem:entropy}}}{\lambda_{\ref{lem:entropy}}}\bigg),
\end{align}
where the final inequality uses $(1-\ol{W}_i(w_i))/\ol{W}_i(w_i)\le 1/\lambda_{\ref{lem:entropy}}$ at any $\lambda_{\ref{lem:entropy}}$-robust step $i$.

Now, let $I \subseteq [\ell]$ denote the steps $i$ which are $\lambda_{\ref{lem:entropy}}$-robust, and let $J \subseteq I$ denote the steps $i$ which are \emph{not} $R$-drops (so that $I \setminus J$ is the set of $\lambda_{\ref{lem:entropy}}$-robust $R$-drops). Our goal is to provide a lower bound on $|J|$.

Since $h_0\le\snorm{f}_\infty\le\snorm{f}_1 = 1$ (this uses the $1$-Lipschitz condition on $\log_{2}f$), we have
\begin{align}
\label{eq:telescoping-bound}
    L(N\sqrt{n})^{-1} &\le \prod_{i \in I \setminus J}\frac{h_i}{h_{i-1}} \le \prod_{i \in I\setminus J}(\ol{W}_i(w_i)+(1-\ol{W}_i(w_i))\gamma_{\ref{lem:entropy}}) \nonumber \\
    &= \frac{\prod_{i \in I}(\ol{W}_i(w_i)+(1-\ol{W}_i(w_i))\gamma_{\ref{lem:entropy}})}{\prod_{i \in J}(\ol{W}_i(w_i)+(1-\ol{W}_i(w_i))\gamma_{\ref{lem:entropy}})} \nonumber \\
    &\le \frac{(1+\gamma_{\ref{lem:entropy}}/\lambda_{\ref{lem:entropy}})^{|I|}\prod_{i\in I}\ol{W}_i(w_i)}{\lambda_{\ref{lem:entropy}}^{|J|}} \nonumber \\
    &= (1+\gamma_{\ref{lem:entropy}}/\lambda_{\ref{lem:entropy}})^{|I|}\lambda_{\ref{lem:entropy}}^{-|J|}\prod_{i\in[\ell]}\ol{W}_i(w_i)\prod_{i\in[\ell]\setminus I}\ol{W}_i(w_i)^{-1} \nonumber \\
    & = (1+\gamma_{\ref{lem:entropy}}/\lambda_{\ref{lem:entropy}})^{|I|} \cdot \lambda_{\ref{lem:entropy}}^{-|J|} \cdot \binom{\ell}{\vec{s}}^{-1} \cdot \prod_{i\in[\ell]\setminus I}\ol{W}_i(w_i)^{-1};
\end{align}
here, the first line uses $h_i/h_{i-1} \le 1$ and \cref{eq:ratio-bound}, the third line uses \cref{eq:robust-step-bound}, and the last line uses the identity
\[\prod_{i\in[\ell]}\ol{W}_i(w_i) = \binom{\ell}{\vec{s}}^{-1} := \binom{\ell}{s_1,\dots,s_k}^{-1}.\]
This follows since both sides are equal to the probability that a uniformly random sample from $[k]^{\ell}$, conditioned on having $s_j$ copies of $j$ for each $j\in[k]$, returns $(w_1,\dots,w_\ell)$.

Note that the first and the third terms in the final product in \cref{eq:telescoping-bound} are easy to suitably control (by taking $\gamma_{\ref{lem:entropy}}$ and $\lambda_{\ref{lem:entropy}}$ to be sufficiently small). As we will see next, these parameters also allow us to make the last term at most $\exp(c\epsilon n)$ for any constant $c > 0$. 

Let $K \subseteq [\ell] \setminus I$ denote those indices $i$ such that $\ol{W}_i(w_i) \ge 1-\lambda_{\ref{lem:entropy}}$. Then, 
\begin{align}
\label{eq:product-K}
\prod_{i\in K}\ol{W}_i(w_i)^{-1} \le (1-\lambda_{\ref{lem:entropy}})^{-|K|}.
\end{align}
It remains to bound 
\[\prod_{i\in [\ell]\setminus (I\cup K)} \ol{W}_i(w_i)^{-1}.\]
Note that for every $i \in [\ell]\setminus (I\cup K)$, we have $\ol{W}_i(w_i) \le \lambda_{\ref{lem:entropy}}$. Let $J_j$ for $j\in[k]$ be the set of $i\in[\ell]\setminus(I\cup K)$ with $w_i=j$.

The following is the key point: let $i_1,\dots, i_{u_j} \in J_j$ be all elements of $J_j$ in order. Then, for all $y \in [u_j]$, we have
\begin{align*}
    y \le W_{i_y}(j) \le \lambda_{\ref{lem:entropy}}\ell. 
\end{align*}
Hence, 
\begin{align*}
    u_j \le \lambda_{\ref{lem:entropy}}\ell \quad \text{ and } \quad \ol{W}_{i_y}(w_{i_y})^{-1} \le i_y/y\le \ell/y.
\end{align*}
We derive
\begin{align}
    \label{eq:K-complement-bound}
 \prod_{i\in [\ell]\setminus (I\cup K)}\ol{W}_i(w_i)^{-1} = \prod_{j=1}^k\prod_{i\in J_j}\ol{W}_i(w_i)^{-1} \le \bigg(\prod_{u=1}^{\lceil\lambda_{\ref{lem:entropy}}\ell\rceil}\frac{\ell}{u}\bigg)^k \le \bigg(\frac{e}{\lambda_{\ref{lem:entropy}}}\bigg)^{2k\lambda_{\ref{lem:entropy}}\ell}.   
\end{align}
Substituting \cref{eq:product-K} and \cref{eq:K-complement-bound} in \cref{eq:telescoping-bound}, we have
\begin{align}
\label{eq:J-bound}
    Ln^{-1/2}\exp((\epsilon-H(\vec{p}))n) \le \lambda_{\ref{lem:entropy}}^{-|J|}\cdot \bigg(1+\frac{\gamma_{\ref{lem:entropy}}}{\lambda_{\ref{lem:entropy}}}\bigg)^{\ell}\cdot \binom{\ell}{\vec{s}}^{-1}\cdot (1-\lambda_{\ref{lem:entropy}})^{-\ell}\cdot \bigg(\frac{e}{\lambda_{\ref{lem:entropy}}}\bigg)^{2k\lambda_{\ref{lem:entropy}}\ell}.
\end{align}

We will first choose $\lambda_{\ref{lem:entropy}}$, and then choose some  $\gamma_{\ref{lem:entropy}} < \lambda_{\ref{lem:entropy}}^{2}$. Note that, enforcing the constraint $\gamma_{\ref{lem:entropy}} < \lambda_{\ref{lem:entropy}}^{2}$, we can choose $\lambda_{\ref{lem:entropy}}$ sufficiently small depending on $\epsilon$ and $\xi$ so that the second term, the fourth term, and the fifth term in the product in \cref{eq:J-bound} are each bounded above by $\exp(\epsilon n/10)$ and so that (using Stirling's approximation) the third term is bounded above by $\exp(\epsilon n/10-H(\vec{p})n)$. Hence, we can choose $\lambda_{\ref{lem:entropy}}$ depending on $\epsilon$ and $\xi$ such that
$$n^{-1/2}\exp(\epsilon n/2) \le \lambda_{\ref{lem:entropy}}^{-|J|}.$$
Now, for all $n$ sufficiently large depending on $\epsilon$, we can find $\gamma_{\ref{lem:entropy}}$ sufficiently small depending on $\epsilon, \lambda_{\ref{lem:entropy}}$ such that 
$|J| \ge \gamma_{\ref{lem:entropy}}n$. This completes the proof.
\end{proof}

We are now ready to prove \cref{prop:rough-Linfty}.
\begin{proof}[Proof of \cref{prop:rough-Linfty}]
We use \cref{lem:entropy} along with a union bound. For controlling individual events in the union, we will use the following. Consider a step record $(w_i)_{i=1}^{\ell}$.  We write $A_i=A_{i,0}\cup A_{i,1}$, where each of these is an integer interval of size at least $N$ (this is possible by the admissibility of $\mc{A}$). Now suppose step $i$ is $\lambda_{\ref{lem:entropy}}$-robust with respect to $(w_i)_{i=1}^{\ell}$. If $i > \delta_0 n$, then for any $t \in \mb{R}$, $j\in[k]$ and $z\in\on{supp}(\xi-\xi')\setminus\{0\}$, by \cref{lem:slice-levy-kolmogorov-rogozin}, we have 
\begin{align*}
\mb{E}[f_{\mc{A},W_i-e_j,i-1}(t+zX_i)|X_1,\ldots,X_{i-1}]
&= \frac{1}{|A_i|}\sum_{\tau\in t+zA_i}f_{\mc{A},W_i-e_j,i-1}(\tau)\\
&\le\max_{y\in\{0,1\}}\frac{1}{|A_{i,y}|}\sum_{\tau\in t+zA_{i,y}}f_{\mc{A},W_i-e_j,i-1}(\tau)\\
&\le\max_{y\in\{0,1\}}\frac{2^{|z|}}{|A_{i,y}|}\bigg|\int_{t+z\min A_{i,y}}^{t+z\max A_{i,y}}f_{\mc{A},W_i-e_j,i-1}(\tau)d\tau\bigg|\\
&\le\frac{2^{|z|+1}C_{\ref{lem:slice-levy-kolmogorov-rogozin}}(\lambda_{\ref{lem:entropy}}/2, \xi, \delta_0,K_1,K_2)\max(|z||A_{i,y}|,N)}{|A_{i,y}|N\sqrt{n}}\\
&\le\frac{4^{|z|+1}C_{\ref{lem:slice-levy-kolmogorov-rogozin}}(\lambda_{\ref{lem:entropy}}/2, \xi, \delta_0,K_1,K_2)}{N\sqrt{n}}.
\end{align*}
Here, we have used that $i-1 \ge \delta_0 n$, that $W_i - e_j$ has at least one coordinate in $[\lambda_{\ref{lem:entropy}}(i-1)/2, (1-\lambda_{\ref{lem:entropy}}/2)(i-1)]$ (since $W_i$ satisfies a similar property with coordinate $w_i$), and that each $A_{i,y}$ is length at least $N$. We also used that $\log_2 f$ is $1$-Lipschitz in the second inequality (where the absolute values are put just in case $z < 0$ and the limits of integration are in the wrong direction).

Now, consider $t \in \mb{R}$ with averaging sequence $(t_i)_{i=0}^{\ell}$ and step record $(w_i)_{i=1}^{\ell}$. Note that, given the `starting point' $t_0$ of the averaging sequence, the points $t_1,\dots, t_{i-1}$ are determined by $X_1,\dots,X_{i-1}$. In particular, the event that step $i$ is not an $R$-drop is determined by $t_0, X_1,\dots, X_i, w_1,\dots,w_i$.  
Therefore, by Markov's inequality, we see that for any $\lambda_{\ref{lem:entropy}}$-robust step $i$ with $i > \delta_0 n$, given the step record $(w_i)_{i=1}^{\ell}$ and the starting point $t_0$ of the averaging sequence $(t_i)_{i=0}^{\ell}$,
\begin{align}
\label{eq:R-drop-bound}
    \mb{P}[ \text{step } i \text{ is not an }R\text{-drop} | X_1,\dots, X_{i-1}] \le \frac{k^34^{2\snorm{\vec{a}}_\infty+1}C_{\ref{lem:slice-levy-kolmogorov-rogozin}}(\lambda_{\ref{lem:entropy}}/2, \xi, \delta_0 ,K_1,K_2)}{R}. 
\end{align}
This follows from a union bound over the at most $k^3$ possible conditions for an $R$-drop and the fact that all $z\in\on{supp}(\xi-\xi')\setminus\{0\}$ have magnitude at most $2\snorm{\vec{a}}_\infty$.

From here on, the proof closely follows the proofs of \cite[Proposition~4.5]{Tik20}. Fix parameters as given in the proposition statement. Let $\lambda_{\ref{lem:entropy}} = \lambda_{\ref{lem:entropy}}(\xi,\epsilon)$. 
We choose $\gamma_{\ref{prop:rough-Linfty}} = \gamma_{\ref{lem:entropy}}(\xi,\epsilon)$. Further, we set  $R' = \gamma_{\ref{lem:entropy}}L/2$, where $L \ge 1$ will be chosen later.  

Let $\mc{E}_L$ denote the event that $\snorm{f_{\mc{A},\vec{s},\ell}}_\infty\ge L(N\sqrt{n})^{-1}$. For $(X_1,\dots,X_n) \in \mc{E}_L$, by \cref{lem:entropy}, there exists $t \in \mb{R}$ with $f_{\mc{A},\vec{s},\ell}(t) \ge L(N\sqrt{n})^{-1}$ with averaging sequence $(t_i)_{i=0}^{\ell}$ and step record $(w_i)_{i=1}^{\ell}$ such that
\[\#\{i\in[\ell]: \text{step }i\text{ is }\lambda_{\ref{lem:entropy}}\text{-robust and is not a }2R'\text{-drop in }(t_i)_{i=0}^\ell\}\ge\gamma_{\ref{lem:entropy}} n.\]
We then shift $t_0$ to the nearest integer $\wt{t_0}$. We also shift $(t_i)_{i=1}^{\ell}$ by the same amount to obtain points $(\wt{t_i})_{i=1}^{\ell}$ (note that these points are not necessarily integers). We call the sequence $(\wt{t_i})_{i=0}^{\ell}$, which technically may no longer be an averaging sequence, a \emph{witnessing sequence}. We see that every index which is not a $2R'$-drop in $(t_i)_{i=0}^\ell$ will not be an $R'$-drop in $(\wt{t_i})_{i=0}^\ell$ as $\log_2f$ is $1$-Lipschitz.

Taking a union bound over the choice of the step record is not costly, and note that given $(X_1,\dots,X_n)$ and the step record, the witnessing sequence is completely determined by its starting point $\wt{t_0}$. Furthermore, the definition of the witnessing sequence and the definition of $f_{\mc{A},\vec{s},\ell}$ easily show that
\[\wt{t_0}\in\{\tau\in\mb{Z}: f(\tau) > (2N\sqrt{n})^{-1}\} =: \mc{D}.\]
Note that $\mc{D}$ is a deterministic set depending only on $f$. Further, since $\snorm{f}_1 = 1$ and $\log_2f$ is $1$-Lipschitz, we see that
\[|\mc{D}|\le 4N\sqrt{n}.\]
To summarize, we have shown that if $(X_1,\dots,X_n) \in \mc{E}_L$, then there exists a witnessing sequence $(\wt{t_i})_{i=0}^{\ell}$ with step record $(w_i)_{i=1}^{\ell}$ such that $\wt{t_0} \in\mc{D}$, and such that
\[\#\{i\in[\ell]: \text{step }i\text{ is }\lambda_{\ref{lem:entropy}}\text{-robust and is not an }R'\text{-drop in }(\wt{t_i})_{i=0}^\ell\}\ge\gamma_{\ref{lem:entropy}} n.\]
Therefore, by the union bound and since $N\le k^n$ (as $H(\vec{p})\le\log k$), it follows that
\begin{align*}
    \mb{P}[\mc{E}_L] 
    &\le (2k^2)^n\sup_{\substack{I\subseteq[\ell], |I|=\lceil\gamma_{\ref{lem:entropy}}n\rceil\\\wt{t_0}\in\mc{D}, (w_i)_{i=1}^\ell\in[k]^\ell}}\mb{P}[\text{The witnessing sequence starts at }\wt{t_0}\text{, has step record }(w_i)_{i=1}^\ell\text{, and}\\
&\qquad\qquad\qquad\qquad\qquad\qquad\text{every }i\in I\text{ is }\lambda_{\ref{lem:entropy}}\text{-robust and is not an }R'\text{-drop}],
\end{align*}
where the supremum is only over those $(w_i)_{i=1}^\ell$ which have $s_j$ coordinates equal to $j$ for all $j\in[k]$.

From \cref{eq:R-drop-bound}, taking $\delta_0 = \gamma_{\ref{lem:entropy}}/2$, it follows that the probability appearing on the right hand side above is bounded by
\[\bigg(\frac{2k^34^{2\snorm{\vec{a}}_\infty+1}C_{\ref{lem:slice-levy-kolmogorov-rogozin}}(\lambda_{\ref{lem:entropy}}/2,\xi,\gamma_{\ref{lem:entropy}}/2,K_1,K_2)}{\gamma_{\ref{lem:entropy}}L}\bigg)^{\gamma_{\ref{lem:entropy}}n/2},\]
since there are at least  $\gamma_{\ref{lem:entropy}}n/2$ values of $i \in I$ with $i > \delta_0 n$ and since $R' = \gamma_{\ref{lem:entropy}}L/2$ by definition. Therefore, taking $L$ and $n$ sufficiently large depending on $M$ and the parameters appearing above gives the desired conclusion.
\end{proof}

\subsection{Refining the initial estimate}\label{sub:refining}
We now need to remove the dependence of $L$ on $M$. This is accomplished by the main result of this subsection, \cref{prop:multislice-Linfty-decrement}, which is a multislice and $L^1(\mb{R})$ analogue of \cite[Proposition~4.10]{Tik20}. Even though we are working in the much more complicated setting of real multislices, remarkably, our proof of \cref{prop:multislice-Linfty-decrement} is able to use \cite[Proposition~4.10]{Tik20} as a black box: 
roughly, we first use a re-randomization procedure to reduce smoothing on the multislice for $L^1(\mb{R})$ to smoothing on the hypercube, also for $L^1(\mb{R})$. At this juncture, the necessary smoothing estimate on the hypercube for $L^1(\mb{R})$ can in fact be lifted from the smoothing estimate for the hypercube for $\ell^1(\mb{Z})$, proved in \cite{Tik20}. In particular, we reduce the smoothing estimate for general log-Lipschitz functions in $L^1(\mb{R})$ to that of a simpler class of ``step'' functions, which in turn is equivalent to $\ell^1(\mb{Z})$.

A simpler version of this argument, for Boolean slices, is presented in \cite[Proposition~3.10]{JSS20discrete1}.

\begin{proposition}\label{prop:multislice-Linfty-decrement}
Fix a discrete distribution $\xi$. There exists $h = h(\xi) \ge 1$ so that the following holds. For any $\epsilon\in(0,1)$, $\wt{R}\ge 1$, $L_0\ge h\wt{R}$, and $M\ge 1$, there is $\gamma_{\ref{prop:multislice-Linfty-decrement}} = \gamma_{\ref{prop:multislice-Linfty-decrement}}(\xi)$ and there are $n_{\ref{prop:multislice-Linfty-decrement}} = n_{\ref{prop:multislice-Linfty-decrement}}(\xi,\epsilon,L_0,\wt{R},M) > 0$ and $\eta_{\ref{prop:multislice-Linfty-decrement}} = \eta_{\ref{prop:multislice-Linfty-decrement}}(\xi,\epsilon,L_0,\wt{R},M)\in(0,1)$ with the following property. Let $L_0\ge L\ge h\wt{R}$, let $n\ge n_{\ref{prop:multislice-Linfty-decrement}}$, $N\in\mb{N}$, and let $g\in L^1(\mb{R})$ be a nonnegative function satisfying
\begin{enumerate}[\indent\indent(A)]
    \item $\snorm{g}_1 = 1$,
    \item $\log_2 g$ is $\eta_{\ref{prop:multislice-Linfty-decrement}}$-Lipschitz,
    \item $\int_{t\in I}g(t)\le\wt{R}/\sqrt{n}$ for any interval $I$ of size $N$, and
    \item $\snorm{g}_\infty\le L/(N\sqrt{n})$.
\end{enumerate}
For each $i\le 2\lfloor\epsilon n\rfloor$, let $Y_i$ be a random variable uniform on some disjoint union of integer intervals of cardinality at least $N$ each, and assume that $Y_1,\ldots,Y_{2\lfloor\epsilon n\rfloor}$ are mutually independent. Define a random function $\wt{g}\in L^1(\mb{R})$ by
\[\wt{g}(t) = \mb{E}_b\bigg[g\bigg(t+\sum_{i=1}^{2\lfloor\epsilon n\rfloor}b_iY_i\bigg)\bigg|\#\{b_i=a_j\}=s_j\forall j\in[k]\bigg]\]
where $b=(b_1,\dots,b_{2\lfloor \epsilon n \rfloor})$ is a vector of independent $\xi$ components and $\vec{s}\in\mb{Z}_{\ge 0}^k$ satisfies $\snorm{\vec{s}}_1=2\lfloor\epsilon n\rfloor$ and
\[\norm{\frac{\vec{s}}{2\lfloor\epsilon n\rfloor}-\vec{p}}_\infty\le\gamma_{\ref{prop:multislice-Linfty-decrement}}.\]
Then
\[\mb{P}\bigg[\snorm{\wt{g}}_\infty > \frac{19L/20}{N\sqrt{n}}\bigg]\le\exp(-Mn).\]
\end{proposition}

We now state an analogue of \cref{prop:multislice-Linfty-decrement} for independent scaled Bernoulli random variables, which in fact is strong enough to imply \cref{prop:multislice-Linfty-decrement}.

\begin{proposition}\label{prop:indep-Linfty-decrement}
Fix $h \ge 1$, and let $z\in[h^{-1},h]$. For any $\epsilon\in(0,1)$, $\wt{R}\ge 1$, $L_0\ge 64h^2\wt{R}$, and $M\ge 1$, there are $n_{\ref{prop:indep-Linfty-decrement}} = n_{\ref{prop:indep-Linfty-decrement}}(h,\epsilon,L_0,\wt{R},M) > 0$ and $\eta_{\ref{prop:indep-Linfty-decrement}} = \eta_{\ref{prop:indep-Linfty-decrement}}(h,\epsilon,L_0,\wt{R},M)\in(0,1)$ with the following property. Let $L_0\ge L\ge 64h^2\wt{R}$, let $n\ge n_{\ref{prop:indep-Linfty-decrement}}$, $N\in\mb{N}$, and let $g\in L^1(\mb{R})$ be a nonnegative function satisfying
\begin{enumerate}[\indent\indent(A)]
    \item $\snorm{g}_1 = 1$,
    \item $\log_2 g$ is $\eta_{\ref{prop:indep-Linfty-decrement}}$-Lipschitz,
    \item $\int_{t\in I}g(t)\le\wt{R}/\sqrt{n}$ for any interval $I$ of size $N$, and
    \item $\snorm{g}_\infty\le L/(N\sqrt{n})$.
\end{enumerate}
For each $i\le\lfloor\epsilon n\rfloor$, let $Y_i$ be a random variable uniform on some disjoint union of integer intervals of cardinality at least $N$ each, and assume that $Y_1,\ldots,Y_{\lfloor\epsilon n\rfloor}$ are mutually independent. Define a random function $\wt{g}\in L^1(\mb{R})$ by
\[\wt{g}(t) = \mb{E}_bg\bigg(t+z\sum_{i=1}^{\lfloor\epsilon n\rfloor}b_iY_i\bigg)\]
where $b$ is a vector of independent $\on{Ber}(1/2)$ components. Then
\[\mb{P}\bigg[\snorm{\wt{g}}_\infty\ge\frac{9L/10}{N\sqrt{n}}\bigg]\le\exp(-Mn).\]
\end{proposition}

This follows almost immediately from an $\ell^\infty(\mb{Z})$ decrement result established by Tikhomirov \cite{Tik20}.

\begin{proposition}[{\cite[Proposition~4.10]{Tik20}}]\label{prop:linfty-decrement}
For any $p\in(0,1/2]$, $\epsilon\in(0,1)$, $\wt{R}\ge 1$, $L_0\ge 16\wt{R}$, and $M\ge 1$ there are $n_{\ref{prop:linfty-decrement}} = n_{\ref{prop:linfty-decrement}}(p,\epsilon,L_0,\wt{R},M) > 0$ and $\eta_{\ref{prop:linfty-decrement}} = \eta_{\ref{prop:linfty-decrement}}(p,\epsilon,L_0,\wt{R},M)\in(0,1)$ with the following property. Let $L_0\ge L\ge 16\wt{R}$, let $n\ge n_{\ref{prop:linfty-decrement}}$, $N\in\mb{N}$, and let $g\in\ell^1(\mb{Z})$ be a nonnegative function satisfying
\begin{enumerate}[\indent\indent(A)]
    \item $\snorm{g}_1 = 1$,
    \item $\log_2 g$ is $\eta_{\ref{prop:linfty-decrement}}$-Lipschitz,
    \item $\sum_{t\in I}g(t)\le\wt{R}/\sqrt{n}$ for any integer interval $I$ of size $N$, and
    \item $\snorm{g}_\infty\le L/(N\sqrt{n})$.
\end{enumerate}
For each $i\le\lfloor\epsilon n\rfloor$, let $Y_i$ be a random variable uniform on some disjoint union of integer intervals of cardinality at least $N$ each, and assume that $Y_1,\ldots,Y_{\lfloor\epsilon n\rfloor}$ are mutually independent. Define a random function $\wt{g}\in\ell^1(\mb{Z})$ by
\[\wt{g}(t) = \mb{E}_bg\bigg(t+\sum_{i=1}^{\lfloor\epsilon n\rfloor}b_iY_i\bigg)\]
where $b$ is a vector of independent $\on{Ber}(p)$ components. Then
\[\mb{P}\bigg[\snorm{\wt{g}}_\infty > \frac{(1-p(1-1/\sqrt{2}))L}{N\sqrt{n}}\bigg]\le\exp(-Mn).\]
\end{proposition}
\begin{remark}
In \cite[Proposition~4.10]{Tik20}, there is a condition $N\le 2^n$ which is not necessary (indeed, it is not used anywhere in the proof) and so has been dropped. In fact, we will only need values $N\le k^n$, in which case one can actually replace $n$ by $kn$ and $\epsilon$ by $\epsilon/k$ (and adjust other parameters appropriately) in order to deduce what we need directly from the statement as written in \cite{Tik20}. We will only need this statement for $p=1/2$.
\end{remark}
Now we first prove \cref{prop:indep-Linfty-decrement}.
\begin{proof}[Proof of \cref{prop:indep-Linfty-decrement}]
Consider the operator $\mc{O}: L^1(\mb{R})\to\mb{\ell}^1(\mb{Z})$ given by
\[(\mc{O}\omega)(t) = \int_{-z/2}^{z/2}\omega(zt+u)~du.\]\
We note that $\snorm{\omega}_1 = \snorm{\mc{O}\omega}_1$ and if $\omega$ is nonnegative and $\log_2\omega$ is $\eta$-Lipschitz, then
\[z2^{-\eta h/2}\snorm{\omega}_\infty\le\snorm{\mc{O}\omega}_\infty\le z\snorm{\omega}_\infty.\]

Given $g\in L^1(\mb{R})$ satisfying the given conditions, we consider $g'\in\ell^1(\mb{Z})$ defined via
\[g' = \mc{O}g.\]
We see that $g'$ satisfies properties (A), (B), (C), (D) of \cref{prop:linfty-decrement} with log-Lipschitz constant slightly changed (depending on $z$, hence $h$), $L$ changed to $zL$, and $\wt{R}$ increased to $4h\wt{R}$. These last changes are responsible for the condition $L_0\ge 64h^2\wt{R}$.

Since $zL\ge h^{-1}L\ge 16(4h\wt{R})$, we may apply \cref{prop:linfty-decrement} to $g'$ to deduce that $\snorm{\wt{g'}}_\infty$ is small, except with superexponentially small probability. Here $\wt{g'}$ is averaged in the sense of \cref{prop:linfty-decrement} with respect to the same $Y_1,\ldots,Y_{\lfloor\epsilon n\rfloor}$.

Now, by Fubini's theorem, note that
\[\wt{g'} = \mc{O}\wt{g},\]
where $\wt{g}$ is averaged in the sense of \cref{prop:indep-Linfty-decrement}. Therefore, 
\[\mb{P}\bigg[\snorm{\mc{O}\wt{g}}_\infty > \frac{(2+\sqrt{2})zL/4}{N\sqrt{n}}\bigg]\le\exp(-Mn),\]
so that
\[\mb{P}\bigg[\snorm{\wt{g}}_\infty > \frac{(2+\sqrt{2})2^{\eta_{\ref{prop:indep-Linfty-decrement}} h/2}L/4}{N\sqrt{n}}\bigg]\le\exp(-Mn).\]
Finally, if $\eta_{\ref{prop:indep-Linfty-decrement}}$ is appropriately small, we deduce the desired as
\[\frac{2+\sqrt{2}}{4} < \frac{9}{10}.\qedhere\]
\end{proof}
Finally, we are able to deduce \cref{prop:multislice-Linfty-decrement}.
\begin{proof}[Proof of \cref{prop:multislice-Linfty-decrement}]
Similar to the proof of \cref{lem:mslkr}, we can use an equivalent method of sampling from the $\vec{s}$-multislice to rewrite $\wt{g}(t)$ as
\begin{align*}
\wt{g}(t) &= \mb{E}_b\bigg[g\bigg(t+\sum_{i=1}^{\lfloor\epsilon n\rfloor}(b_{2i-1}Y_{2i-1}+b_{2i}Y_{2i})\bigg)\bigg|\#\{b_i=a_j\}=s_j \forall j \in [k]\bigg]\\
&= \mb{E}_{b,b'}\bigg[g\bigg(t+\sum_{i=1}^{\lfloor\epsilon n\rfloor}b_{2i-1}Y_{2i-1}+b_{2i}Y_{2i} + b_i'(b_{2i}-b_{2i-1})(Y_{2i-1}-Y_{2i})\bigg)\bigg|\#\{b_i=a_j\}=s_j \forall j \in [k]\bigg],
\end{align*}
where $b'$ is an $\lfloor \epsilon n \rfloor$-dimensional vector with \emph{independent} $\on{Ber}(1/2)$ components. Below, we will fix $b$ and use only the randomness in $b'$. In order to do this, let
\[B_0 := \bigg\{b_1,\dots, b_{2\lfloor \epsilon n \rfloor} : \#\{i : b_{2i-1}=a_1,b_{2i}=a_2\}\ge\min(\vec{p})^2\epsilon n/8\bigg\}.\]
Then, provided that $\gamma_{\ref{prop:multislice-Linfty-decrement}}$ is chosen sufficiently small depending on $\xi$, and $n$ is sufficiently large depending on $\xi$ and $\epsilon$, we have 
\[\mb{E}_{b}[1_{B_0}|\#\{b_i=a_j\}=s_j \forall j \in [k]] > \frac{1}{2}.\]
Let $\mc{E}_L$ denote the event (depending on $Y_1,\dots, Y_{2\lfloor \epsilon n \rfloor}$) that $\snorm{\wt{g}}_\infty > 19L/(20N\sqrt{n})$. Now, suppose $Y_1,\dots, Y_{2\lfloor \epsilon n \rfloor} \in \mc{E}_L$, and suppose further that $\snorm{\wt{g}}_{\infty}$ is attained at $t \in \mb{R}$.
Let 
\[B_{1} := \bigg\{b_1,\dots, b_{2\lfloor \epsilon n \rfloor}: \mb{E}_{b'}\bigg[g\bigg(t+\sum_{i=1}^{\lfloor\epsilon n\rfloor}(b_{2i-1}Y_{2i-1}+b_{2i}Y_{2i} + b_i'(b_{2i}-b_{2i-1})(Y_{2i-1}-Y_{2i})\bigg)\bigg|b\bigg]\ge\frac{9L/10}{N\sqrt{n}}\bigg\}.\]
Since $\snorm{g}_\infty\le L/(N\sqrt{n})$, it follows from the reverse Markov inequality that  
\[\mb{E}_{b}[1_{B_1}| \#\{b_i=a_j\}=s_j \forall j \in [k]] > \frac{1}{2}.\]
Thus, we see that for every $(Y_1,\dots, Y_{2\lfloor \epsilon n \rfloor}) \in \mc{E}_L$, there exists some $b \in B_0 \cap B_1$. Hence, taking a union bound, we see that
\begin{align}
\label{eq:union-B_0}
\mb{P}&\bigg[\snorm{\wt{g}}_\infty > \frac{19L/20}{N\sqrt{n}}\bigg] \le \mb{P}[\exists b \in B_0 : b\in B_1]
 \nonumber \\
&\le|B_0|\sup_{b\in B_0}\mb{P}\bigg[\exists t: \mb{E}_{b'}\bigg[g\bigg(t+\sum_{i=1}^{\lfloor\epsilon n\rfloor}(b_{2i-1}Y_{2i-1}+b_{2i}Y_{2i} + b_i'(b_{2i}-b_{2i-1})(Y_{2i-1}-Y_{2i}))\bigg)\bigg]\ge\frac{9L/10}{N\sqrt{n}}\bigg]\nonumber\\
&\le|B_0|\sup_{b\in B_0}\mb{P}\bigg[\exists t: \mb{E}_{b'}\bigg[g\bigg(t+\sum_{i=1}^{\lfloor\epsilon n\rfloor} b_i'(b_{2i}-b_{2i-1})(Y_{2i-1}-Y_{2i})\bigg)\bigg]\ge\frac{9L/10}{N\sqrt{n}}\bigg].
\end{align}
We now bound the probability appearing on the right hand side of the above equation uniformly for $b \in B_0$. We fix $b \in B_0$, and note that, by definition, there is a set $I = \{i_1,\dots, i_m\} \subseteq \lfloor \epsilon n \rfloor$ such that $|I| = m \ge\min(\vec{p})^2\epsilon n/8$ and such that for all $j \in [m]$,
\begin{align*}
b_{i_j}'(b_{2{i_j}}-b_{2{i_j-1}})(Y_{2i_j-1}-Y_{2i_j})&= b_{i_j}'(a_2-a_1)(Y_{2i_j-1}-Y_{2i_j}).
\end{align*}
For $j \in [k]$, let $Y_j^{b}:= Y_{2i_j} - Y_{2i_j - 1}$. Let $Y_{-2\cdot I}$ denote all components of $Y_1,\dots,Y_{2\lfloor \epsilon n \rfloor}$, except those corresponding to indices in $2\cdot I$, and let $Y_{2\cdot I}$ denote the remaining components. Then, for $b \in B_0$ and a choice of $Y_{-2\cdot I}$, we define the random function (depending on $Y_{2\cdot I}$),
\[\wt{g}_{b,Y_{-2\cdot I}}(t) := \mb{E}_{b'}g\bigg(t+(a_1-a_2)\sum_{j=1}^{\lfloor \min(\vec{p})^{2}\epsilon n/8\rfloor}b_j'Y_j^b\bigg).\]
Thus, we see that for any $b \in B_0$ and $Y_{-2\cdot I}$, the probability appearing on the right hand side of \cref{eq:union-B_0} is bounded by
\[\mb{P}\bigg[\snorm{\wt{g}_{b, Y_{-2\cdot I}}}_{\infty} \ge \frac{9L/10}{N\sqrt{n}} \bigg],\]
where the probability is over the choice of $Y_{2\cdot I}$. 

At this point, we can apply \cref{prop:indep-Linfty-decrement} to $\wt{g}_{b,Y_{-2\cdot I}}$. Let us quickly check that the hypotheses of \cref{prop:indep-Linfty-decrement} are satisfied. The assumptions on $g$ needed in \cref{prop:indep-Linfty-decrement} are satisfied because the same properties are assumed in \cref{prop:multislice-Linfty-decrement} (see below for the log-Lipschitz condition). Moreover, $b_1',\dots, b'_{\lfloor \min(\vec{p})^{2}\epsilon n/8 \rfloor}$ are independent $\on{Ber}(1/2)$ random variables. Finally, notice that, given $Y_{-2\cdot I}$, each $Y_j^{b}$ is a random variable uniform on some disjoint intervals of cardinality at least $N$ each (since $Y_j^{b}$ is a translation of $Y_{2i_j}$ which is assumed to satisfy this property). Also, $a_1-a_2$ is bounded away from $0$ (in terms of $\xi$).

Thus, \cref{prop:indep-Linfty-decrement} shows that the expression on the right hand side of \cref{eq:union-B_0} is bounded above by
\begin{align*}
|B_0|\sup_{b\in B_0, Y_{-2\cdot I} }\mb{P}\bigg[\snorm{\wt{g}_{b, Y_{-2\cdot I}}}_\infty > \frac{9L/10}{N\sqrt{n}}\bigg]\le k^n\exp(-M\min(\vec{p})^2n/8),
\end{align*}
provided that we choose $\eta_{\ref{prop:multislice-Linfty-decrement}}$ sufficiently small compared to $\eta_{\ref{prop:indep-Linfty-decrement}}(d, \on{min}(\vec{p})^2\epsilon/8, L_0, \tilde{R}, M)$, where $d = \max(|a_2 - a_1|, |a_2 - a_1|^{-1})$. The desired result now follows after rescaling $M$ by a constant factor (depending on $\xi$).
\end{proof}

\subsection{Deriving the final result}\label{sub:deriving} \cref{thm:inversion-of-randomness} is almost immediate given the previous propositions. 
\begin{proof}[Proof of \cref{thm:inversion-of-randomness}]
The proof of \cref{thm:inversion-of-randomness} given \cref{prop:rough-Linfty,prop:multislice-Linfty-decrement} is similar to the derivation in \cite[Theorem~4.2]{Tik20}; we refer the reader to the companion note \cite[Theorem~3.1]{JSS20discrete1} for a leisurely account of the complete details. In words, we use \cref{prop:rough-Linfty} to show that any revelation of the first $(1-\gamma)n$ random variables (for suitably chosen $\gamma$) will give an $L^\infty$ bound at roughly the right scale $L'(N\sqrt{n})^{-1}$, where $L'$ depends on $M$. Then, we divide the remaining random variables into a constant number of pieces (this number depends on $L'$ and $C_{\ref{lem:mslkr}}$ with appropriate parameters), and use \cref{prop:multislice-Linfty-decrement} iteratively on these pieces. As in Steps 2 and 3 in the proof of \cite[Theorem~3.1]{JSS20discrete1}, we may restrict our attention to the case when each of the remaining pieces is a well-conditioned multislice (this automatically implies that the first $(1-\gamma)n$ coordinates are also a well-conditioned multislice), noting that the complement of this event occurs with very low probability, and the product of this probability with $L'(N\sqrt{n})^{-1}$ is much smaller than the bound on the essential sup-norm that we are aiming for.
\end{proof}

\subsection{Independent model}\label{sub:independent}
We conclude this section with an analogue of \cref{cor:threshold-inversion} in the independent case. 

\begin{theorem}\label{cor:independent-threshold-inversion}
Fix a discrete distribution $\xi$. For $0 < \delta < 1/4$, $K_3 > K_2 > K_1 > 1$, $\epsilon\ll \snorm{\vec{p}}_\infty$, and a given parameter $M\ge 1$, there is  $L_{\ref{cor:independent-threshold-inversion}} = L_{\ref{cor:independent-threshold-inversion}}(\xi,\epsilon,\delta,K_1,K_2,K_3) > 0$ independent of $M$ and $n_{\ref{cor:independent-threshold-inversion}} = n_{\ref{cor:independent-threshold-inversion}}(\xi,\epsilon,\delta,K_1,K_2,K_3,M)\ge 1$ such that the following holds.

Let $n\ge n_{\ref{cor:independent-threshold-inversion}}$, $1\le N\le\snorm{\vec{p}}_\infty^{-n}\exp(-\epsilon n)$ and $\mc{A}$ be $(N,n,K_1,K_2,K_3,\delta)$-admissible. Then
\[\bigg| \bigg\{x\in \mc{A}: \mc{L}_{\xi}\bigg(\sum_{i=1}^nb_ix_i,\sqrt{n}\bigg)\ge L_{\ref{cor:independent-threshold-inversion}}N^{-1}\bigg\}\bigg|\le e^{-Mn}|\mc{A}|.\]
\end{theorem}

The proof of \cref{cor:independent-threshold-inversion} is analogous to that of \cref{thm:inversion-of-randomness} followed by  \cref{cor:threshold-inversion}, except that the random variables $b_i$ are now independent copies of $\xi$. This independence simplifies matters dramatically, as one can derive an analogue of \cref{prop:rough-Linfty} by simply considering drops (as in \cite[Proposition~4.5]{Tik20}) instead of ``well-conditioned'' drops,  and then using subsampling arguments similar to those appearing above to prove analogues of \cref{prop:multislice-Linfty-decrement} and \cref{thm:inversion-of-randomness}. We leave the details to the interested reader.

\section{Non-almost-constant vectors}\label{sec:incompressible}

\subsection{Preliminaries}\label{sec:preliminaries}
We recall the following basic definition of almost-constant vectors.
\begin{definition}\label{def:almost-constant}
We define $\on{Cons}(\delta,\rho)$ to be the set of $x\in\mb{S}^{n-1}$ such that there is $\lambda\in\mb{R}$ so that $|x_i-\lambda|\le\rho/\sqrt{n}$ for at least $(1-\delta)n$ values $i\in[n]$. We let $\on{Noncons}(\delta,\rho) = \mb{S}^{n-1}\setminus\on{Cons}(\delta,\rho)$.
\end{definition}
The following is a standard fact regarding non-almost-constant vectors. For the reader's convenience, we provide complete details of the proof in \cite{JSS20discrete1}.

\begin{lemma}[{\cite[Lemma~2.3]{JSS20discrete1}}]
\label{lem:nonconstant-admissible}
For $\delta,\rho \in (0,1/4)$, there exist $\nu,\nu' > 0$ depending only on $\delta, \rho$, and a finite set $\mc{K}$ of positive real numbers, also depending only only on $\delta, \rho$, such that the following holds.  

If $x\in \on{Noncons}(\delta, \rho)$, then at least one of the following two conclusions is satisfied.
\begin{enumerate}
\item There exist $\kappa, \kappa' \in \mc{K}$ such that 
\[|x_i| \le \frac{\kappa}{\sqrt{n}} \text{ for at least } \nu n \text{ indices }i\in [n], \text{ and }\]
\[\frac{\kappa + \nu'}{\sqrt{n}} < |x_i| \le \frac{\kappa'}{\sqrt{n}} \text{ for at least } \nu n \text{ indices }i\in [n].\]
\item There exist $\kappa, \kappa' \in \mc{K}$ such that
\[\frac{\kappa}{\sqrt{n}} < x_i < \frac{\kappa'}{\sqrt{n}} \text{ for at least } \nu n \text{ indices }i\in [n], \text{ and }\]
\[-\frac{\kappa'}{\sqrt{n}} < x_i < -\frac{\kappa}{\sqrt{n}} \text{ for at least } \nu n \text{ indices }i\in [n].\]
\end{enumerate}
\end{lemma}

We record two useful tensorization statements.
\begin{lemma}[{\cite[Lemma~3.2]{Tik20}}]\label{lem:tensorization}
Let $\chi_1,\ldots,\chi_m$ be independent random variables.
\begin{itemize}
    \item Assume that for all $\epsilon\ge \epsilon_0$, 
    \[\mb{P}[|\chi_i|\le\epsilon]\le K\epsilon.\]
    Then for $\epsilon\ge \epsilon_0$
    \[\mb{P}[\snorm{(\chi_1,\ldots,\chi_m)}_2\le \epsilon\sqrt{m}]\le(CK\epsilon)^m,\]
    where $C$ is an absolute constant. 
    \item Assume that for some $\eta, \tau >0$, 
    \[\mb{P}[|\chi_i|\le\eta]\le \tau.\]
    Then for $\epsilon\in (0,1]$,
    \[\mb{P}[\snorm{(\chi_1,\ldots,\chi_m)}_2\le \eta\sqrt{\epsilon m}]\le(e/\epsilon)^{\epsilon m}\tau^{m-\epsilon m}.\]
\end{itemize}
\end{lemma}

We will need a standard concentration estimate for the operator norm of a random matrix with independent centered sub-Gaussian entries.
\begin{lemma}[{\cite[Lemma~4.4.5]{Ver18}}]\label{lem:subgaussian-norm}
There exists an absolute constant $C>0$ such that the following holds. Let $A$ be an $m\times n$ i.i.d.~matrix with mean $0$, sub-Gaussian entries with sub-Gaussian norm at most $K$. Then for any $t\ge 0$ we have
\[\mb{P}[\snorm{A}\le C(\sqrt{m}+\sqrt{n} + t)]\le 2\exp(-t^2/K^2).\]
\end{lemma}

Finally, we will need the following lemma, proved using randomized rounding (cf.~\cite{Liv18}), which is a straightforward generalization of \cite[Lemma~5.3]{Tik20}. We omit details since the proof is identical.
\begin{lemma}\label{lem:round}
Let $y = (y_1,\ldots,y_n)\in \mb{R}^n$ be a vector, and let $\mu>0$, $\lambda\in \mb{R}$ be fixed. Let $\Delta$ denote a probability distribution which is supported in $[-s,s]^{n}$. There exist constants $c_{\ref{lem:round}}$ and $C_{\ref{lem:round}}$, depending only on $s$, for which the following holds. 
Suppose that for all $t\ge\sqrt{n}$,
\[\mb{P}\bigg[\bigg|\sum_{i=1}^nb_iy_i-\lambda\bigg|\le t\bigg]\le \mu t,\]
where $(b_1,\dots, b_n)$ is distributed according to $\Delta$. Then, there exists a vector $y' = (y'_1,\dots,y'_n) \in \mb{Z}^{n}$ satisfying
\begin{enumerate}[(R1)]
    \item $\snorm{y-y'}_\infty\le 1$,
    \item $\mb{P}[|\sum_{i=1}^nb_iy_i'-\lambda|\le t]\le C_{\ref{lem:round}}\mu t$ for all $t\ge \sqrt{n}$,
    \item $\mc{L}(\sum_{i=1}^nb_iy_i',\sqrt{n})\ge c_{\ref{lem:round}}\mc{L}(\sum_{i=1}^nb_iy_i,\sqrt{n})$,
    \item $|\sum_{i=1}^ny_i-\sum_{i=1}^ny_i'|\le C_{\ref{lem:round}}\sqrt{n}$.
\end{enumerate}
\end{lemma}

\subsection{The structure theorem}\label{sub:structure-theorem}

Let $\xi$ be a discrete distribution, and let $A = A(\xi)$ denote an $(n-1) \times n$ random matrix, each of whose entries is an independent copy of a $\xi$ random variable. We fix a function $v(A)$ which takes as input an $(n-1)\times n$ matrix and outputs a unit vector in its right kernel. A key ingredient in the proof of the main result of this section, \cref{thm:incompressible}, is a structure theorem for kernel vectors of $A$, which encodes the fact that (with very high probability)  non-almost-constant kernel vectors of $A$ must be maximally unstructured in the relevant sense.

\begin{definition}\label{def:conditional-threshold}
Fix a discrete distribution $\xi$. Let $\vec{\gamma}\in\mb{R}_{\ge 0}^k$ with $\snorm{\vec{\gamma}}_\infty < \min(\vec{p})$, and let $L \ge 1$. Then, for any integer $n \ge 1$ and $x \in \mb{S}^{n-1}$, we define
\[\mc{T}_{\xi,\vec{\gamma}}(x,L) := \sup\bigg\{t \in (0,1): \mc{L}_{\xi,\vec{\gamma}}\bigg(\sum_{i=1}^{n}b_i x_i, t\bigg) > Lt\bigg\}.\]
We also define
\[\mc{T}_{\xi}(x,L) := \sup\bigg\{t \in (0,1): \mc{L}_{\xi}\bigg(\sum_{i=1}^{n}b_i x_i, t\bigg) > Lt\bigg\}.\]
\end{definition}

\begin{proposition}
\label{prop:structure}
Let $\delta, \rho, \epsilon \in (0,1)$. There exist $L_{\ref{prop:structure}} = L_{\ref{prop:structure}}(\delta, \rho, \xi, \epsilon)$, $\gamma_{\ref{prop:structure}} = \gamma_{\ref{prop:structure}}(\delta, \rho, \xi, \epsilon)$ and $n_{\ref{prop:structure}} = n_{\ref{prop:structure}}(\delta, \rho, \xi, \epsilon)$ such that for all $n \ge n_{\ref{prop:structure}}$, with probability at least $1-k^{-2n}$, exactly one of the following holds.
\begin{itemize}
    \item $v(A) \in \on{Cons}(\delta, \rho)$, or
    \item $\mc{T}_{\xi, \gamma_{\ref{prop:structure}}1_k}(v(A), L_{\ref{prop:structure}}) \le \exp((\epsilon-H(\vec{p}))n)$.
\end{itemize}
\end{proposition}

The proof of \cref{prop:structure} follows from \cref{cor:threshold-inversion} and \cref{lem:round}. At a high level, we consider dyadic intervals for the threshold function of $v(A)$, round $v(A)$ to an appropriate integer lattice (using \cref{lem:round}), and then use \cref{cor:threshold-inversion}. In order to be in a setting where we can apply \cref{cor:threshold-inversion}, we use an idea of Litvak and Tikhomirov \cite{LT20} -- using the Chernoff bound, we restrict ourselves to the event that all but $O_{\xi,\gamma_{\ref{cor:threshold-inversion}}}(1)$ rows belong to a well-conditioned multislice corresponding to $\xi$. We refer the reader to \cite[Proposition~4.2]{JSS20discrete1} for further details. 

\subsection{Invertibility on the bulk}\label{sub:non-almost-constant}

We are now ready to state and prove the key result of this section.
\begin{theorem}\label{thm:incompressible}
Fix a discrete distribution $\xi$. For any $\delta,\rho,\epsilon > 0$, there exists $C_{\ref{thm:incompressible}} = C_{\ref{thm:incompressible}}(\xi, \delta, \rho, \epsilon) > 0$ and $n_{\ref{thm:incompressible}}(\xi, \delta, \rho, \epsilon) \ge 1$ such that for all $n \ge n_{\ref{thm:incompressible}}$ and $t\ge 0$,
\[\mb{P}\bigg[\inf_{x\in\on{Noncons}(\delta,\rho)}\snorm{M_n(\xi)x}_2\le t/\sqrt{n}\wedge\inf_{y\in\on{Cons}(\delta,\rho)}\snorm{yM_n(\xi)}_2 > C_{\ref{thm:incompressible}}t\bigg]\le C_{\ref{thm:incompressible}}t + \exp((\epsilon-H(\vec{p}))n).\]
\end{theorem}

\begin{proof}
Let $M := M_n(\xi)$ for simplicity, and let $\delta,\rho,\epsilon > 0$ be as in the statement of the theorem. We will denote the columns of $M$ by $M^{(1)},\dots, M^{(n)}$. Also, for each $i \in [n]$, $M^{(-i)}$ denotes the subspace spanned by all columns of $M$ except for $M^{(i)}$. 

\textbf{Step 1: } Let $\gamma = \gamma_{\ref{prop:structure}}(\delta, \rho, \xi, \epsilon)$. Let $W_{\gamma}\subseteq\on{supp}(\xi)^{n}$ denote the set of vectors $x\in\on{supp}(\xi)^{n}$ such that $\#\{x_i=a_j\} \in [p_jn - \gamma n, p_jn+\gamma n]$ for all $j \in [k]$. Let $Q \ge 1$ be a constant such that the event 
\[\mc{W}_{Q}: = \{ |\{i \in [n]: M^{(i)} \notin W_{\gamma}| \le Q \} \]
holds with probability at least $1-k^{-1729n}$. Then, it suffices to bound
\begin{equation}\label{eq:desired}\mb{P}\bigg[\inf_{x\in\on{Noncons}(\delta,\rho)}\snorm{Mx}_2\le t/\sqrt{n}\wedge\inf_{y\in\on{Cons}(\delta,\rho)}\snorm{yM}_2 > Ct\wedge \mc{W}_Q\bigg].\end{equation}
Let us denote the first of the three events in the equation above by $\mc{E}_R$, and the second event by $\mc{E}_L$.

Let $x = x(M)$ denote a vector in $\on{Noncons}(\delta, \rho)$ certifying the event $\mc{E}_R$, so that
\[\snorm{x_1M^{(1)}+\cdots+x_nM^{(n)}}_2\le t/\sqrt{n}.\]
Using \cref{lem:nonconstant-admissible}, there is a set $I \subseteq [n]$ such that $|I|\ge \nu n$ and such that for all $i\in I$, $|x_i| \ge \kappa/\sqrt{n}$, for some $\kappa:=\kappa(\delta, \rho) > 0$. In particular, since for any $i\in [n]$,
$\snorm{x_1M^{(1)} + \dots + x_nM^{(n)}}_{2} \ge |x_i|\on{dist}(M^{(i)}, M^{(-i)}),$ it follows that
\[\on{dist}(M^{(i)}, M^{(-i)}) \le \frac{t}{\kappa}\text{ for all }i \in I.\]
Also, on the event $\mc{W}_{Q}$, there are at least $\nu n/2$ indices $i \in I$ such that $M^{(i)} \in W_{\gamma}$. Thus, we see that
\[\cref{eq:desired} = \mb{P}[\mc{E}_R\wedge\mc{E}_L\wedge\mc{W}_Q]\le \frac{2}{\nu n}\sum_{i=1}^n\mb{P}[\on{dist}(M^{(i)}, M^{(-i)}) \le t/\kappa \wedge \mc{E}_L \wedge M^{(i)} \in W_{\gamma}].\]

\textbf{Step 2: }By symmetry, it suffices to bound $\mb{P}[\mc{M}_1]$, where 
\[\mc{M}_1 := \on{dist}(M^{(1)}, M^{(-1)}) \le t/\kappa \wedge \mc{E}_L \wedge M^{(1)} \in W_{\gamma}.\]
Let $v(M^{(-1)})$ be a unit vector normal to $M^{(-1)}$. Then, by \cref{prop:structure}, except with probability $k^{-2n}$ (over the randomness of $M^{(-1)}$), exactly one of the following holds.
\begin{itemize}
    \item $v(M^{(-1)}) \in \on{Cons}(\delta, \rho)$, or
    \item $\mc{T}_{\xi, \gamma 1_k}(v(M^{(-1)}), L) \le\exp((\epsilon-H(\vec{p}))n),$
\end{itemize}
where $L := L_{\ref{prop:structure}}(\delta, \rho, \xi, \epsilon)$. If the first possibility occurs, then $\mc{M}_1$ cannot hold as $v(M^{(-1)}) \in \on{Cons}(\delta, \rho)$ satisfies
\[\snorm{v(M^{(-1)})M}_{2} = |\langle M^{(1)}, v(M^{(-1)})\rangle | \le \on{dist}(M^{(1)}, M^{(-1)}) \le t/\kappa\le Ct,\]
(choosing $C$ appropriately), which contradicts $\mc{E}_L$. Hence, the second possibility must hold. But then, using $\on{dist}(M^{(1)}, M^{(-1)}) \ge |\langle M^{(1)}, v(M^{(-1)}) \rangle|$, we have that (over the randomness of $M^{(1)}$),
\begin{align*}
    \mb{P}[\on{dist}(M^{(1)}, M^{(-1)}) \le t/\kappa \wedge M^{(1)} \in W_{\gamma}] 
    &\le \mb{P}[|\langle M^{(1)}, v(M^{(-1)})\rangle| \le t/\kappa \mid M^{(1)} \in W_{\gamma}]\\
    &\le \frac{Lt}{\kappa} + \exp((\epsilon-H(\vec{p}))n)).\qedhere
\end{align*}

\end{proof}

\section{Preliminary invertibility estimates}\label{sec:initial-compressible}
In this short section, we will prove a version of \cref{thm:main-approximate} with the weaker singularity estimate $(\snorm{\vec{p}}_{\infty} + o_n(1))^{n}$. This estimate, which generalizes \cite[Theorem~A]{Tik20}, will be used crucially in our refined treatment of invertibility for almost-constant vectors in the next section. The techniques in this section also serve as a gentle warm-up to the next section, where much more involved versions of the arguments are presented.

We begin with the following elementary fact regarding sums of $\xi$ random variables. 
\begin{lemma}\label{lem:p-infty}
Fix a discrete distribution $\xi$. There is $\theta = \theta(\xi) > 0$ such that for all $x\in\mb{S}^{n-1}$, 
\[\mc{L}_{\xi}(b_1x_1+\cdots+b_nx_n,\theta)\le\snorm{\vec{p}}_\infty.\]
\end{lemma}
\begin{proof}
This is essentially identical to the proof given in \cite[Lemma~3.5]{Tik20}. Briefly, if $\snorm{x}_\infty\ge\delta$, then we can choose $\theta$ small enough (depending on $\delta$ and $\xi$) so the claim is immediate. Otherwise $\snorm{x}_\infty < \delta$ and $\snorm{x}_2 = 1$, in which case the claim follows from \cref{lem:LKR} as long as $\delta$ is sufficiently small depending on $\xi$. 
\end{proof}
Combining the above estimate with the second part of \cref{lem:tensorization}, we have the following. 
\begin{corollary}\label{cor:initial-fixed}
Fix a discrete distribution $\xi$. For every $\epsilon > 0$, there exists $c > 0$ depending on $\epsilon$ and $\xi$ such that for any $x\in\mb{S}^{n-1}$ and $y\in\mb{R}^n$, we have
\[\mb{P}[\snorm{M_n(\xi)x-y}_2\le c\sqrt{n}]\le (\snorm{p}_\infty+\epsilon)^n.\]
\end{corollary}

Moreover, combining this corollary with the low metric entropy of $\on{Cons}(\delta, \rho)$ and \cref{lem:subgaussian-norm}, we obtain the following (weak) estimate for invertibility on almost-constant vectors. 

\begin{corollary}\label{cor:initial-compressible}
Fix a discrete distribution $\xi$. For every $\epsilon > 0$, there exist $\delta,\rho,c > 0$ depending on $\epsilon$ and $\xi$ such that for any $y\in\mb{R}^n$,
\[\mb{P}\bigg[\inf_{x\in\on{Cons}(\delta,\rho)}\snorm{M_n(\xi)x-y}_2\le c\sqrt{n}\bigg]\le (\snorm{p}_\infty+\epsilon)^n.\]
\end{corollary}

Next, we show that with very high probability, the inverse of any fixed vector is unstructured.
\begin{proposition}\label{prop:initial-unstructured}
Fix a discrete distribution $\xi$. For every $\epsilon,\eta > 0$, there exist $\delta,\rho,L > 0$ depending on $\epsilon, \eta, \xi$ such that for any $y\in\mb{R}^n$,
\[\mb{P}\bigg[\exists x\in\mb{S}^{n-1}: M_n(\xi)x\parallel y\wedge(x\in\on{Cons}(\delta,\rho)\vee\mc{T}_{\xi}(x,L)\ge(\snorm{\vec{p}}_\infty+\eta)^n\bigg]\le(\snorm{\vec{p}}_\infty+\epsilon)^n.\]
\end{proposition}
\begin{proof}
This follows essentially from combining \cref{cor:initial-compressible} with a cruder analogue of \cref{prop:structure}, the only difference being that we are considering $M_{n}(\xi)x \parallel y$ for arbitrary $y \in \mb{R}^{n}$ as opposed to only $y = 0$.

To handle this last point, we begin by choosing (using \cref{lem:subgaussian-norm}) a sufficiently large constant $K$ so that $\mc{E}_K = \{\snorm{M_n(\xi) - \mb{E}[\xi]J_{n\times n}}\le K\sqrt{n}\}$ satisfies $\mb{P}[\mc{E}_K^c]\le\snorm{\vec{p}}_\infty^{2n}$. Then, it suffices to restrict to $\mc{E}_K$. Moreover, by the triangle inequality, we see that on the event $\mc{E}_K$, $\snorm{M_n(\xi)} \le K\sqrt{n} + \mb{E}[\xi]n$, so that in particular, on the event in the proposition (intersected with $\mc{E}_K$), we have that $M_{n}(\xi)x = ty_0$ with $y_0 \in \mb{S}^{n-1}$ fixed and for some $t \in \mb{R}$ with $|t| \le K\sqrt{n} + \mb{E}[\xi]n$.

Now, for the treatment of vectors in $\on{Cons}(\delta, \rho)$, we can divide the range of $t$ into $n^{3}$ uniformly spaced intervals, apply \cref{cor:initial-compressible} with $y$ equal to the mid-point of an interval times $y_0$, and use the union bound. For the treatment of vectors in $\on{Noncons}(\delta, \rho) \wedge \mc{T}(x,L) \le (\snorm{\vec{p}}_{\infty} + \epsilon)^{n}$, we divide the range of $t$ into $\snorm{\vec{p}}_{\infty}^{-2n}$ equally spaced intervals, use a slight generalization of the argument in the proof \cite[Proposition~4.2]{JSS20discrete1} with $M$ sufficiently large (depending on $\xi$) also for $y$ equal to the mid-point of an interval times $y_0$, and finally use the union bound. We leave the details to the interested reader. \qedhere  
\end{proof}

Using \cref{cor:initial-compressible} and \cref{prop:initial-unstructured}, we can prove the following weaker version of \cref{thm:main-approximate}.

\begin{theorem}\label{thm:weak}
Let $\xi$ be a discrete random variable. For any $\epsilon > 0$, there exist $C, n_0 > 0$ depending on $\xi, \epsilon$ such that for all $n \ge n_0$ and $t\ge 0$,
\[\mb{P}[s_n(M_n)\le t/\sqrt{n}]\le C t + (\snorm{\vec{p}}_\infty+\epsilon)^n.\]
\end{theorem}
\begin{proof}
The deduction of this theorem follows from the argument in \cite[Section~5]{Tik20} with the application of \cref{cor:initial-compressible} and \cref{prop:initial-unstructured} at the appropriate steps. A more complicated version of this deduction also appears in \cref{sec:deduction-main}, so we omit details.
\end{proof}

\section{Almost-constant vectors}\label{sec:compressible}
The goal of this section is to prove \cref{thm:compressible}. The proof is presented at the end of the section and needs a few intermediate steps. 

For the proof, we will need to isolate the following natural class of almost-elementary vectors.
\begin{definition}(Almost-elementary vectors) For $\delta > 0$ and $i,j \in [n]$, $i\neq j$, let
\begin{align*}
    \on{Elem}_i(\delta) &:= \{x \in \mb{S}^{n-1}: \snorm{x-e_i}_{2} \le \delta\},\\
    \on{Elem}_{i,j}(\delta) &:= \{x \in \mb{S}^{n-1}: \snorm{x-(e_i-e_j)/\sqrt{2}}_{2} \le \delta\},\\
    \on{Elem}'_{i,j}(\delta) &:= \{x \in \mb{S}^{n-1}: \snorm{x-(e_i + e_j)/\sqrt{2}}_{2} \le \delta\}.  
\end{align*}
Also, let
\begin{align*}
    \on{Elem}(\delta) &:= \bigcup_{i \in [n]}\on{Elem}_i(\delta) \cup \bigcup_{i,j \in [n], i\neq j}\on{Elem}_{i,j}(\delta),\\
    \on{Elem}'(\delta) &:= \on{Elem}(\delta) \cup \bigcup_{i,j \in [n], i \neq j}\on{Elem}'_{i,j}(\delta).
\end{align*}
\end{definition}

For excluding almost-constant vectors which are not almost-elementary, we will need to develop sharp results regarding the L\'evy concentration function of discrete random variables. 
\begin{proposition}\label{prop:p-characterization}
Fix a discrete distribution $\xi$ and $\delta\in(0,1/2)$. There exists $\theta = \theta(\delta,\xi) > 0$ such that for all $x\in\mb{S}^{n-1} \setminus \on{Elem}'(\delta)$, 
\[\mc{L}_{\xi}(b_1x_1+\cdots+b_nx_n,\theta)\le\snorm{\vec{p}}_2^2-\theta.\]
\end{proposition}
\begin{proof}
Since $\on{Elem}'(\delta)$ is increasing with $\delta$, it suffices to prove the statement for sufficiently small $\delta$ (depending on $\xi$), which will be chosen during the course of the proof. Moreover, we may assume that $|x_1|\ge |x_2|\ge \dots \ge |x_n|$.   

Since $x \notin \on{Elem}_{1}(\delta)$, we must have $\snorm{(x_2,\dots, x_n)}_{2} \ge \delta /2$. In case $|x_2| \le \delta^{4}$, then we are done using \cref{lem:LKR} (cf.~\cite[Lemma~2.6]{JSS20discrete1}) for all sufficiently small $\delta$. Similarly, if $\snorm{(x_3,\dots,x_n)}_{2} \ge \delta/4$ and $|x_3|\le \delta^{4}$, we are done. We now analyze the remaining situations via case analysis. 

\textbf{Case I: }$\delta^{4} \le  |x_2| < (1-\delta^5)|x_1|$. Since 
$\mc{L}_{\xi}(b_1x_1+\dots + b_n x_n, \theta) \le \mc{L}_{\xi}(b_1 x_1 + b_2 x_2, \theta)$, it suffices to bound the latter. Let $\xi'$ be an independent copy of $\xi$. For any $s \in \mb{R}$, we have
\begin{align*}
\mb{P}[x_1 \xi + x_2 \xi'\in [s-c,s+c]]^{2} 
&= \bigg(\sum_{a}\mb{P}[\xi' = a]\mb{P}[|\xi - x_1^{-1}(s-x_2a)|\le c|x_1|^{-1}]\bigg)^{2}\\
&\le\bigg(\sum_a\mb{P}[\xi'=a]^2\bigg)\bigg(\sum_a\mb{P}[|\xi - x_1^{-1}(s-x_2a)|\le c|x_1|^{-1}]^2\bigg)\le\snorm{\vec{p}}_{2}^{4},
\end{align*}
where the sum is over $a\in\on{supp}(\xi)$. Here, the equality is by definition, the first inequality is Cauchy-Schwarz, and the last inequality holds as long as $c > 0$ is chosen small enough in terms of $\delta,\xi$. Let us elaborate on this final point. We choose $c > 0$ small enough so that  $c|x_1|^{-1}\le c\delta^{-4}$ is smaller than $|x_2/x_1|$ times half the minimum gap in $\on{supp}(\xi)$, which is possible since $|x_2/x_1| \ge \delta^{4}$. Now, such a choice of $c$ clearly implies that each summand in $\sum_a\mb{P}[|\xi - x_1^{-1}(s-x_2a)|\le c|x_1|^{-1}]^2$ covers at most a single atom in $\on{supp}(\xi)$, and that different choices of $a, a' \in \on{supp}(\xi)$ cover distinct atoms in $\on{supp}(\xi)$.

Moreover, for such a choice of $c$, equality in the final inequality holds if and only if there is a permutation $\sigma$ on $\on{supp}(\xi)$ such that for all $a \in \on{supp}(\xi)$,
\[\mb{P}[\xi' = \sigma(a)] = \mb{P}[|\xi - x_1^{-1}(s-x_2 a)|\le c|x_1|^{-1}].\]
Summing over all the atoms in $\on{supp}(\xi)$, we see that if equality holds in the final inequality, then 
\[\on{supp}(\xi)\subseteq\bigcup_{j=1}^k[x_1^{-1}(s-x_2 a)-c|x_1|^{-1},x_1^{-1}(s-x_2 a)+c|x_1|^{-1}],\]
so that in particular, $\on{supp}(\xi)$ is contained in an interval of length at most
$|x_2/x_1|m_\xi+2c|x_1|^{-1}$, where $m_\xi = \max\on{supp}(\xi)-\min\on{supp}(\xi)$. But since $|x_2/x_1|\le 1-\delta^5$ and $c|x_1|^{-1}\le c\delta^{-4}$, we see (by taking $c > 0$ sufficiently small) that $\on{supp}(\xi)$ is contained in an interval of length at most $(1-\delta^5/2)m_\xi$, which contradicts the definition of $m_\xi$. Hence, we see that equality cannot hold in the final inequality. 

Since equality does not hold, it follows from the above discussion that (for $c > 0$ sufficiently small), we have the stronger inequality
\[\mb{P}[x_1 \xi + x_2 \xi'\in [s-c,s+c]]^2\le\snorm{\vec{p}}_2^2(\snorm{\vec{p}}_2^2-(\min\vec{p})^2),\]
which completes the analysis in this case, noting that the choice of $c$ depends only on $\xi,\delta$.

\textbf{Case II: }$|x_2| \ge (1-\delta^5)|x_1|, \snorm{(x_3,\dots, x_n)}_{2} \le \delta/4$. This implies that $x \in \on{Elem}'_{1,2}(\delta) \cup \on{Elem}_{1,2}(\delta)$, thereby violating our assumption.

\textbf{Case III: }$\delta^{4} \le |x_3|\le (1-\delta^5)|x_1|$. This can be treated in exactly the same way as Case I. 

\textbf{Case IV: }$(1-\delta^5)|x_1|\le |x_3| \le |x_2|$ and $|x_2|\ge\delta^4$. It suffices to bound $\mc{L}_{\xi}(b_1 x_1 + b_2 x_2 + b_3 x_3)$. Let $u_i\in\{\pm 1\}$ be defined via $u_i = \on{sgn}(x_i) = x_i/|x_i|$. Let $m'_\xi > 0$ be the smallest positive real such that $\on{supp}(\xi) \subseteq [-m'_\xi, m'_\xi]$. 

We begin by noting that for any $s\in\mb{R}$,
\begin{align*}
\mb{P}[x_1 \xi_1 + x_2 \xi_2 + x_3 \xi_3&\in [s-c,s+c]]\\
&= \mb{P}[|x_1| (u_1\xi_1 + |x_1|^{-1}|x_2|u_2 \xi_2 + |x_1|^{-1}|x_3|u_3 \xi_3)\in [s-c,s+c]]\\ 
&\le\mb{P}[|x_1|(u_1\xi_1+u_2\xi_2+u_3\xi_3)\in[s-c-3\delta^5m'_\xi,s+c+3\delta^5m'_\xi]],
\end{align*}
where the inequality uses $(1-\delta^5) \le |x_1|^{-1}|x_3| \le |x_1|^{-1}|x_2| \le 1$, $|x_1| \le 1$, and the definition of $m'_\xi$. 

Since $|x_1|\ge|x_2|\ge\delta^4$, this localizes the value of $u_1\xi_1+u_2\xi_2+u_3\xi_3$ to an interval of length at most $2(c\delta^{-4}+3\delta m'_\xi)$. As discussed at the beginning, we can assume that $\delta$ is sufficiently small based on $\xi$. By first choosing $\delta > 0$ sufficiently small depending on $\xi$, and then choosing $c > 0$ sufficiently small depending on $\delta$ and $\xi$, we may assume that $2(c\delta^{-4}+3\delta m'_\xi)$ is smaller than the minimum distance between two distinct atoms in both $\on{supp}(\xi+\xi'+\xi'')$ and $\on{supp}(\xi+\xi'-\xi'')$, where $\xi, \xi', \xi''$ are independent copies of $\xi$. Note that, after possibly multiplying by an overall negative sign, $u_1\xi_1+u_2\xi_2+u_3\xi_3$ is distributed as either $\xi + \xi' + \xi''$ or $\xi + \xi' - \xi''$. 

Therefore, by our choice of $\delta$ and $c$, we see that it suffices to show that for all $s \in \mb{R}$,
\[\mb{P}[\xi_1+\xi_2+\xi_3=s] \le \snorm{\vec{p}}_2^2 - c_\xi,\quad\mb{P}[\xi_1+\xi_2-\xi_3=s] \le \snorm{\vec{p}}_2^2 - c_\xi,\]
for some $c_\xi > 0$ depending only on $\xi$.
Now for $u_3\in\{\pm 1\}$, we have
\begin{align*}
\mb{P}[\xi_1 + \xi_2 + u_3\xi_3 = s]^2
&= \bigg(\sum_a\mb{P}[\xi_3 = a]\mb{P}[\xi_1 +\xi_2 = s - u_3a]\bigg)^2\\
&\le\bigg(\sum_a\mb{P}[\xi_3 = a]^2\bigg)\bigg(\sum_a\mb{P}[\xi_1 +\xi_2 = s - u_3a]^2\bigg)\\
&\le\bigg(\sum_a\mb{P}[\xi_3 = a]^2\bigg)\bigg(\sum_{a'\in\on{supp}(\xi_1+\xi_2)}\mb{P}[\xi_1 +\xi_2 = a']^2\bigg)\\
&\le\snorm{\vec{p}}_{2}^{4},
\end{align*}
where the first line is by definition, the second line is Cauchy-Schwarz, and the last line follows by Young's convolution inequality. To obtain the inequality with a positive constant $c_\xi > 0$, we note that equality cannot hold in the third line since $\on{supp}(\xi_1+\xi_2)$ has strictly more positive atoms than $\on{supp}(\xi)$ (since $\xi$ is supported on at least $2$ points), and this leads to the desired improvement since $\xi$ has finite support.
\end{proof}

When $\xi$ is not a translate of an origin-symmetric distribution, the above result can be strengthened. 
\begin{proposition}\label{prop:p-characterization-II}
Fix a discrete distribution $\xi$ and $\delta \in (0,1/2)$. Suppose that $\xi$ is not a translate of any origin-symmetric distribution. Then, there exists $\theta = \theta(\delta,\xi) > 0$ such that for all $x \in \mb{S}^{n-1}\setminus \on{Elem}(\delta)$,
\[\mc{L}_\xi(b_1x_1+\cdots+b_nx_n,\theta)\le\snorm{\vec{p}}_2^2-\theta.\]
\end{proposition}
\begin{proof}
As before, since $\on{Elem}(\delta)$ is increasing with $\delta$, it suffices to prove the statement for sufficiently small $\delta$ depending on $\xi$. By \cref{prop:p-characterization}, we can choose $\theta = \theta(\delta, \xi) > 0$ such that for all $x \in \mb{S}^{n-1} \setminus \on{Elem}(\delta)$,
\[\mc{L}_\xi(b_1x_1+\cdots+b_nx_n,\theta)\le\snorm{\vec{p}}_2^2-\theta.\]
Hence, it remains to prove the result for $x \in \on{Elem}'(\delta)\setminus \on{Elem}(\delta)$. By symmetry, it suffices to consider  $x\in\on{Elem}'_{1,2}(\delta)$. We will bound $\mc{L}_{\xi}(b_1 x_1 + b_2 x_2, \theta)$. 

We use an argument similar to \textbf{Case IV} of the proof of \cref{prop:p-characterization}. Let $m'_\xi > 0$ be the smallest positive real for which $\on{supp}(\xi) \subseteq [-m'_\xi, m'_\xi]$. We have
\[\mb{P}[x_1\xi_1+x_2\xi_2\in[s-c,s+c]]\le\mb{P}\bigg[\frac{1}{\sqrt{2}}(\xi_1+\xi_2)\in[s-c-2m'_\xi\delta,s+c+2m'_\xi\delta]\bigg].\]
Once again, by choosing $\delta$ and $c$ sufficiently small (depending on $\xi$), we may assume that $2(c + 2m'_\xi \delta)$ is smaller than the minimum distance between two distinct atoms in $\on{supp}(\xi+\xi')$, where $\xi, \xi'$ are independent copies of $\xi$. With this choice of $\delta$ and $c$, the problem reduces to showing that there exists some $c_\xi > 0$ depending only on $\xi$ such that for all $s \in \mb{R}$, \[\mb{P}[\xi_1+\xi_2=s] \le \snorm{\vec{p}}_2^2 - c_\xi.\]
We have
\[\mb{P}[\xi_1+\xi_2=s] = \sum_a\mb{P}[\xi_1=a]\mb{P}[\xi_2=s-a]\le\bigg(\sum_a\mb{P}[\xi_1=a]^2\bigg)^{1/2}\bigg(\sum_a\mb{P}[\xi_2=s-a]^2\bigg)^{1/2}\le\snorm{\vec{p}}_2^2,\]
where the first inequality is Cauchy-Schwarz. To obtain the improved inequality with $c_\xi > 0$, we note that equality can hold in both inequalities if and only if $\mb{P}[\xi_1 = a] = \mb{P}[\xi_2 = s-a]$, which implies that $\xi$ is a shift (by $s/2$) of an origin-symmetric random variable. Since we have assumed that $\xi$ is not a shift of an origin-symmetric random variable, we see that equality cannot hold, and using that the support of $\xi$ is finite, we can conclude. 
\end{proof}

Using the preceding lemmas, and exploiting the low metric entropy of $\on{Cons}(\delta, \rho)$ along with \cref{lem:subgaussian-norm}, we obtain the following corollary. Note that since $\xi$ may not have mean $0$, one must perform the standard trick of densifying the net of these vectors along the direction $1_n$ (see \cite[Proposition~3.6]{Tik20}). We refer the reader to \cite[Proposition~2.7]{JSS20discrete1} for the (standard) details. 
\begin{corollary}
\label{cor:non-elementary}
Fix a discrete distribution $\xi$. For all $\delta' > 0$, there exist $\delta,\rho, \epsilon', n_0 > 0$, depending on $\xi$ and $\delta'$, such that for all $n \ge n_0$,
\[\mb{P}[\exists x \in \on{Cons}(\delta, \rho) \setminus \on{Elem}'(\delta'): \snorm{M_{n}(\xi)x}_{2} \le \epsilon'\sqrt{n}] \le (\snorm{\vec{p}}_{2}^{2} - \epsilon')^{n}.\]
Further, if $\xi$ is not a shift of any origin-symmetric random variable, then the same conclusion holds with $\on{Elem}(\delta')$ instead of $\on{Elem}'(\delta')$. 
\end{corollary}

Given the previous corollary, it remains to analyze vectors in $\on{Elem}'(\delta')$ (or only in $\on{Elem}(\delta')$ if $\xi$ is not a shift of any origin-symmetric random variable), which is the content of the remainder of this section. 

\subsection{Two columns}\label{sub:two-columns}
We first handle vectors in $\on{Elem}_{i,j}(\delta')$. By the invariance of the distribution of $M_n(\xi)$ under permuting columns, it suffices to analyze vectors in $\on{Elem}_{1,2}(\delta')$. We show the following.

\begin{proposition}\label{prop:e1-e2}
Fix a discrete distribution $\xi$. There exist $\delta', \eta, n_0 > 0$ depending on $\xi$ such that for all $n \ge n_0$ and $t \le 1$,
\[\mb{P}[\exists x\in\on{Elem}_{1,2}(\delta'): \snorm{M_n(\xi)x}_2\le t]\le\snorm{\vec{p}}_2^{2n} + (\snorm{\vec{p}}_2^2-\eta)^n + t\exp(-\eta n).\]
\end{proposition}

We will need the following preliminary lemma, which essentially follows from the seminal work of Rudelson and Vershynin \cite{RV08}. Since we were not able to locate the statement we need in the literature, we provide details below and in \cite[Lemma~2.9]{JSS20discrete1}.
\begin{lemma}\label{lem:strange-block}
Fix $S, s > 0$. There exist $C', c', n_0 > 0$ depending on $s, S$ such that the following holds. For all $n\ge n_0$, any $v \in \mb{R}^{n}$ with $\snorm{v}_{2} \ge 1$,  any $\kappa \in (0,1)$, and all $t \le 1$, we have
\[\mb{P}[\exists x \in \mb{R}^{n-1}: \snorm{Ax - v}_{2} \le t] \le C'n^{3}\sqrt{t}\exp(\kappa n) + \exp(-c'n)\exp(\kappa n),\]
where $A$ is an $n\times (n-1)$ random matrix, each of whose entries is an independent random variable with sub-Gaussian norm at most $S$, and such that all but a collection of $\kappa n$ specified entries have variance at least $s$.  
\end{lemma}
\begin{proof}
By the law of total probability, it suffices to assume that the $\kappa n$ specified entries are deterministic, and take the values $a_1,\dots, a_{\kappa n}$. Consider the $n\times (n-1)$ random matrix $A'$, which has the same distribution as $A$, except for the $\kappa n$ specified entries, which are now replaced by $a_1 + b_1,\dots, a_{\kappa n} + b_{\kappa n}$, where $b_1,\dots, b_{\kappa n}$ are independent $\on{Ber}(1/2)$ random variables. 

From a slight generalization of \cite[Lemma~2.9]{JSS20discrete1} (specifically, one should replace the application of \cite{RV08} with an inhomogeneous version due to \cite{LTV19} and replace $2^{-cn}$ by $t$, see the proof of \cref{lem:sharp-block}), we get that there exist $C',c', n_0$ depending on $s,S$ such that for all $n \ge n_0$, for any $v \in \mb{R}^{n}$ with $\snorm{v}_{2} \ge 1$, and for all $t \ge 1$, we have
\[\mb{P}[\exists x \in \mb{R}^{n-1}: \snorm{A'x - v}_{2} \le t] \le C'n^{3}\sqrt{t} + \exp(-c'n).\]
The conclusion now follows since, with probability $2^{-\kappa n}$, $b_{1} = \dots = b_{\kappa n} = 0$. 
\end{proof}

We now prove \cref{prop:e1-e2}.

\begin{proof}[Proof of \cref{prop:e1-e2}]
By \cref{lem:subgaussian-norm}, we can choose $K > 0$ depending on $\xi$ such that $\mb{P}[\mc{E}_K]\le\snorm{\vec{p}}_2^{3n}$, where
\[\mc{E}_K := \{\snorm{M_n(\xi)-\mb{E}[\xi]J_{n\times n}}\le K\sqrt{n}\}.\]
For $\delta'\in(0,1/4)$, which will be chosen later in terms of $\xi$, let
\[\mc{E} := \{\exists x\in\mb{B}_2^n(e_1,\delta')\cap\mb{S}^{n-1}: \snorm{M_n(\xi)Qx}_2\le t\},\]
where $Q$ is the rotation matrix whose bottom-right $(n-2) \times (n-2)$ minor is the identity matrix, and the top-left $2\times 2$ minor is the rotation matrix given by
\[\begin{bmatrix}\frac{1}{\sqrt{2}}&\frac{1}{\sqrt{2}}\\ -\frac{1}{\sqrt{2}}&\frac{1}{\sqrt{2}}\end{bmatrix}.\]
Up to scaling $\delta'$ by a constant factor, this is clearly equivalent to the event that we wish to bound.

Note that on the event $\mc{E}$, there exists some vector $y = e_1+u\in\mb{R}^n$ with $u_1 = 0$ and $\snorm{u}_2\le 4\delta'$ such that 
\[\snorm{M_n(\xi)Qy}_2\le 2t.\]
Let $u' = (u_2,\ldots,u_n) \in \mb{R}^{n-1}$, let $\wt{M}^{(1)}$ be the first column of $M_{n}(\xi)Q$, and let $\wt{M}^{(-1)}$ denote the $n\times (n-1)$ matrix obtained by removing this column. Then, on the event $\mc{E}\wedge \mc{E}_K$, we have
\begin{align*}
\snorm{\wt{M}^{(1)}- \mb{E}[\xi]J_{n\times n-1}u'}_2&\le\snorm{\wt{M}^{(1)}+\wt{M}^{(-1)}u'}_2+\snorm{(\wt{M}^{(-1)}-\mb{E}[\xi]J_{n\times n-1})u'}_2\\
&\le 2t + K\sqrt{n}\cdot 4\delta'\\
&\le 8K\delta'\sqrt{n}
\end{align*}
for all sufficiently large $n$, since $t\le 1$. 

The key point is the following. Let $\Xi := \on{supp}(\xi - \xi')/\sqrt{2} \subseteq \mb{R}$. Let
\[\mc{C} := \{a \in \Xi^{n} : \exists \lambda \in \mb{R} \text{ with } \snorm{a-\lambda 1_{n}}_{2} \le 8K\delta' \sqrt{n} \},\]
and for $\kappa = \kappa(\delta',\xi) > 0$, to be chosen later depending on $\delta',\xi$, and for $z \in \Xi$, let
\[\mc{C}_z: = \mc{C} \cap \{a \in \mb{R}^{n}: |\on{supp}(a-z1_n)| \le \kappa n\}.\]
It is easy to see that
\[\mc{C}\subseteq\bigcup_{z\in\Xi}\mc{C}_z\]
for an appropriate choice of $\kappa$ which goes to $0$ as $\delta'$ goes to $0$. Furthermore,
\[\mb{P}[\wt{M}^{(1)}\in\mc{C}_z]\le\mb{P}[(\xi-\xi')/\sqrt{2}=z]^n\exp(c_{\kappa,\delta',\xi}n),\]
where $c_{\kappa,\delta',\xi} > 0$ goes to $0$ as $\kappa,\delta'$ go to $0$. Therefore, we have
\[\mb{P}[\wt{M}^{(1)} \in \mc{C}] \le \snorm{\vec{p}}_{2}^{2n}\exp(2c_{\kappa,\delta', \xi}n),\text{ and } \]
\[\mb{P}[\wt{M}^{(1)} \in \mc{C}\setminus \mc{C}_0] \le (\snorm{\vec{p}}_{2}^{2} - c_{\xi})^n\]
for some $c_{\xi} > 0$ depending only on $\xi$, provided that $\delta'$ (hence $\kappa$) is chosen sufficiently small. Here, for the second inequality, we have used that by Cauchy-Schwarz (as in the proof of \cref{prop:p-characterization}), the unique most probable atom of $(\xi - \xi')/\sqrt{2}$ is at $0$, and is $\snorm{\vec{p}}_{2}^{2}$, so that any other atom in $\Xi$ has probability at most $\snorm{\vec{p}}_{2}^{2} - 2c_{\xi}$ for some $c_{\xi} > 0$. 

So far, we have shown that for all $\kappa$ and $\delta'$ sufficiently small (depending on $\xi$), we have
\begin{align*}
\mb{P}[\mc{E}]&\le\snorm{\vec{p}}_2^{3n} + \sum_{a\in\mc{C}}\mb{P}[\wt{M}^{(1)}=a]\mb{P}[\exists u'\in\mb{R}^{n-1}: \snorm{\wt{M}^{(-1)}u'+a}_2\le 2t|\wt{M}^{(1)}]\\
&\le\snorm{\vec{p}}_2^{3n} + (\snorm{\vec{p}}_2^2-c_{\xi})^n + \sum_{a\in\mc{C}_0}\mb{P}[\wt{M}^{(1)}=a]\mb{P}[\exists u'\in\mb{R}^{n-1}: \snorm{\wt{M}^{(-1)}u'+a}_2\le 2t|\wt{M}^{(1)}].
\end{align*}

We proceed to bound the third term in the above sum. 

\textbf{Case I: }If $a = 0$, we have $\mb{P}[\wt{M}^{(1)} = 0] = \snorm{p}_{2}^{2n}$. 

\textbf{Case II: }If $a \neq 0$, we have in particular that $\snorm{a}_{2} \ge h_{\xi} > 0$. The crucial observation is the following. Given $\wt{M}^{(1)} = a$, the entries of the first column of $\wt{M}^{(-1)}$ are independent random variables, each of which is distributed as the sum of two i.i.d.~copies of $\xi/\sqrt{2}$, conditioned on knowing their difference. In particular, for the coordinates $i \in [n]$ for which $a_i = 0$, the corresponding coordinate of the first column of $\wt{M}^{(-1)}$ is distributed as $\sqrt{2}\cdot \xi^\ast$, where $\xi^\ast$ has the same support as $\xi$ but takes on atom $a_i$ with probability proportional to $p_i^2$. Thus, we see that conditioned on $\wt{M}^{(1)} = a \in
\mc{C}_0$, all entries of $\wt{M}^{(-1)}$ are independent with sub-Gaussian norm at most $S_{\xi}$, and all but at most $\kappa n$ entries have variance at least $s_{\xi} > 0$. Hence, by \cref{lem:strange-block}, and by using the lower bound $\snorm{a}_{2} \ge h_{\xi}$, we find that there exist $C',c', n_1$ depending on $\xi$  such that for all $n \ge n_{1}$,
\[\mb{P}[\exists u' \in \mb{R}^{n-1}: \snorm{\wt{M}^{(-1)}u'+a}_2\le 2t|\wt{M}^{(1)}] \le C'n^{3}\sqrt{t}\exp(\kappa n) + \exp(-c'n)\exp(\kappa n).\]
Thus, the contribution of this case is at most
\[\snorm{\vec{p}}_{2}^{2n}\exp(2c_{\kappa,\delta', \xi}n)\exp(\kappa n)\bigg(C'n^{3}\sqrt{t} + 2\exp(-c'n)\bigg).\]
By the AM-GM inequality, we have $\snorm{\vec{p}}_{2}^{2n}\sqrt{t} \le t\snorm{\vec{p}}_{2}^{n} + \snorm{\vec{p}}_{2}^{3n}$. The desired conclusion now follows by taking $\eta > 0$ sufficiently small so that $\snorm{\vec{p}}_{2}^{n} \le \exp(-2\eta n)$, and then taking $\delta'$ (hence $\kappa$) sufficiently small so that $2c_{\kappa,\delta', \xi} + \kappa < \min(c'/2, \eta/2)$.
\end{proof}

The preceding proposition handles vectors in $\on{Elem}_{i,j}(\delta)$. If the distribution $\xi$ is a translate of an origin-symmetric distribution, we also need to handle vectors in $\on{Elem}'_{i,j}(\delta)$. In case the distribution $\xi$ is itself an origin-symmetric distribution, the desired bound follows immediately from the previous proposition, using that the distribution of any column of $M_n(\xi)$ is invariant under negation in this case. Therefore, it remains to handle vectors in $\on{Elem}'_{i,j}(\delta)$ when $\xi$ is a nonzero translate of an origin-symmetric distribution, which is done by the next proposition. 

\begin{proposition}\label{prop:e1+e2}
Fix a discrete distribution $\xi$ that is a \emph{nonzero} translate of an origin-symmetric distribution. There exist $\delta', \eta, n_0 > 0$ depending on $\xi$ such that for all $n \ge n_0$ and $t \le 1$,
\[\mb{P}[\exists x\in\on{Elem}_{1,2}'(\delta'): \snorm{M_n(\xi)x}_2\le t]\le (\snorm{\vec{p}}_2^2-\eta)^n + t\exp(-\eta n).\]
\end{proposition}
\begin{proof}
The proof is essentially the same as that of \cref{prop:e1-e2}. The lack of the ``main term'' $\snorm{\vec{p}}_{2}^{2}$ comes from the fact that $e_1+e_2$ is unlikely to be a kernel vector since $\xi$ is not origin-symmetric. 

We quickly discuss the main modifications to the proof of \cref{prop:e1-e2}. Throughout, $s \neq 0$ denotes a real number such that  $\xi$ and $s-\xi$ have the same distribution (such an $s$ exists by our assumption about $\xi$). First, the top-left $2\times 2$ minor of $Q$ is now 
\[\begin{bmatrix}\frac{1}{\sqrt{2}}&-\frac{1}{\sqrt{2}}\\\frac{1}{\sqrt{2}}&\frac{1}{\sqrt{2}}\end{bmatrix}.\]
Next, we let $\Xi = \on{supp}(\xi+\xi')/\sqrt{2}$ and as before, let
\[\mc{C} := \{a\in\Xi^n: \exists\lambda\in\mb{R}\text{ with }\snorm{a-\lambda 1_n}_2\le 8K\delta'\sqrt{n}\}\]
and for $z\in\Xi$,
\[\mc{C}_z := \mc{C}\cap\{a\in\mb{R}^n: |\on{supp}(a-z1_n)|\le\kappa n\},\]
where $\kappa = \kappa(\delta',\xi) > 0$ is chosen as in the previous argument. For such a choice of $\kappa$, we have
\[\mb{P}[\wt{M}^{(1)} \in \mc{C}] \le \snorm{\vec{p}}_{2}^{2n}\exp(2c_{\kappa,\delta', \xi}n),\text{ and } \]
\[\mb{P}[\wt{M}^{(1)} \in \mc{C}\setminus \mc{C}_{s/\sqrt{2}}] \le (\snorm{\vec{p}}_{2}^{2} - c_{\xi})^{n}.\]
This time the inequalities are derived as follows. We note that, by Cauchy-Schwarz, for any $z \in \Xi$, $\mb{P}[\xi + \xi' = z\sqrt{2}] \le \snorm{\vec{p}}_{2}^{2}$, with equality holding if and only if $\mb{P}[\xi = a] = \mb{P}[\xi' = z\sqrt{2} - a]$ for all $a \in \on{supp}(\xi)$, which happens if and only if $z = s/\sqrt{2}$. 

Using this, we have as before that 
\[\mb{P}[\mc{E}]\le \snorm{\vec{p}}_{2}^{3n} + (\snorm{\vec{p}}_2^2-c_\xi)^n+\sum_{a\in\mc{C}_{s/\sqrt{2}}}\mb{P}[\wt{M}^{(1)}=a]\mb{P}[\exists u'\in\mb{R}^{n-1}: \snorm{\wt{M}^{(-1)}u'+a}_2\le 2t|\wt{M}^{(1)}].\]

The most important detail is that for $\kappa \le 1/2$ (say),
every $a\in\mc{C}_{s/\sqrt{2}}$ is nonzero, since it has at least $(1-\kappa)n$ coordinates equal to $s/\sqrt{2}$. Since $s$ is a nonzero constant depending only on $\xi$, we can now use the analysis in
\textbf{Case II} of the proof of \cref{prop:e1-e2}. The final thing to note is that the distribution of the random variable $(\xi - \xi')/\sqrt{2}$, conditioned on $(\xi + \xi')/\sqrt{2} = s/\sqrt{2}$ coincides with the distribution of $(2\xi^\ast-s)/\sqrt{2}$, where $\xi^\ast$ has the same support as $\xi$, but takes on atom $a_i$ with probability proportional to $p_i^2$. The remaining details of the proof are essentially the same.
\end{proof}

\subsection{One column}\label{sub:one-column}
We now handle vectors in $\on{Elem}_i(\delta')$. Once again, by permutation invariance, it suffices to handle $\on{Elem}_{1}(\delta')$. We will prove the following. 

\begin{proposition}\label{prop:e1}
Fix a discrete distribution $\xi$. There exist $C', \delta', \eta, n_0 > 0$ depending on $\xi$ such that for all $n \ge n_0$ and $t\le 1$, 
\[\mb{P}[\exists x\in\on{Elem}_1(\delta'): \snorm{M_n(\xi)x}_2\le t]\le p_0^n + C't\exp(-\eta n) +  (\snorm{\vec{p}}_2^2-\eta)^n.\]
\end{proposition}

The analysis is more delicate than the two column case, since (i) we may have  $p_0 < \snorm{\vec{p}}_\infty$, but we still want  to isolate $p_0$ as the major contribution coming from these events, and (ii) we are aiming for an error term of $(\snorm{\vec{p}}_{2}^{2} - \eta)^{n}$, which may be smaller than $(p_0 - \eta)^{n}$. However, given the preparation above, the rest of the proof is similar to the proof in the sparse Bernoulli case, isolated in \cite[Proposition~2.8]{JSS20discrete1}, except that we need to replace the application of the results of Rudelson and Vershynin \cite{RV08} with the much sharper \cref{prop:initial-unstructured} and \cref{thm:weak}.

We begin with the following proposition. The analogue in the sparse Bernoulli case is \cite[Lemma~2.9]{JSS20discrete1}.

\begin{lemma}\label{lem:sharp-block}
Fix a discrete distribution $\xi$. For any $\eta \in (0,1)$, there exist $C, n_0 > 0$ depending on $\xi, \eta$ for which the following holds.  For any $v\in\mb{R}^n$ with $\snorm{v}_2\ge 1$, we have
\[\mb{P}[\exists x\in\mb{R}^{n-1}: \snorm{Ax-v}_2\le t]\le C\cdot n^3 t ^{1/2} + (\snorm{\vec{p}}_\infty+\eta)^n,\]
where $A$ is a random $n\times(n-1)$ matrix with independent $\xi$ entries.
\end{lemma}
\begin{proof}
Fix $\eta > 0$, and let $\mc{E}$ be the event whose probability we are trying to control. After potentially reindexing the coordinates, we may write
\[A = \begin{bmatrix}R\\A_{n-1}\end{bmatrix},\quad v=\begin{bmatrix}v_1\\v'\end{bmatrix}\]
where $A_{n-1}$ is an $(n-1)\times(n-1)$ matrix and $v'\in\mb{R}^{n-1}$ satisfies $\snorm{v'}_2\ge 1/2$. Let $\mc{E}_S = \{s_{n-1}(A_{n-1})\le\sqrt{t}\}$. By \cref{thm:weak}, we have that for all sufficiently large $n$, there exists a constant $C'$ depending on $\xi$ and $\eta$ such that
\[\mb{P}[\mc{E}_S]\le C'\sqrt{nt} + (\snorm{\vec{p}}_\infty+\eta/2)^n.\]
It therefore suffices to bound the probability of $\mc{E}\wedge\mc{E}_S^c$. In such a situation, we see that $y := (A_{n-1})^{-1}v'$ is unique. Let $y_0 := y/\snorm{y}_2$, and for $\delta, \rho, L$ to be chosen momentarily, let
\[\mc{E}_U = \{y_0\in\on{Cons}(\delta,\rho)\vee\mc{T}_{\xi}(y_0,L)\ge(\snorm{\vec{p}}_\infty+\eta/2)^n\}.\]
By \cref{prop:initial-unstructured}, we can choose $\delta,\rho,L > 0$ depending on $\xi$ and $\eta$ so that
\[\mb{P}[\mc{E}_U]\le(\snorm{\vec{p}}_\infty+\eta/2)^n.\]
Hence, it suffices to bound the probability of $\mc{E}\wedge\mc{E}_S^c\wedge\mc{E}_U^c$. Let $x \in \mb{R}^{n-1}$ be a vector certifying this event. Then, we have for all sufficiently large $n$ that
\[\snorm{A_{n-1}x-v'}_2\le t \implies \snorm{x-y}_2\le t^{1/2}, \text{ and }\]
\[|Rx-v_1|\le t \implies|Ry-v_1|\le t+n t^{1/2}.\]
Furthermore, since $\snorm{v}_{2} \ge 1$, we have $\snorm{y}_2\ge 1/C''n^2$, for some constant $C''$ depending on $\xi$. 

We now fix a realization of $A_{n-1}$ satisfying $\mc{E}_S^{c} \wedge \mc{E}_U^{c}$. In particular, this fixes $y,y_0$ satisfying the conditions in $\mc{E}_U^{c}$ and with $\snorm{y}_{2} \ge 1/C''n^{2}$.   
Now, we use the independence of $R$ and $A_{n-1}$ and the fact that $\mc{E}$ implies
\[|Ry-v_1|\le t +n t^{1/2}\le 2n t^{1/2}.\]
Since
\[\mc{T}_{\xi}(y_0,L) < (\snorm{\vec{p}}_\infty+\eta/2)^n\]
and $\snorm{y}_2\ge 1/C''n^2$, we find that the desired probability is bounded by
\[2LC
''n^3 t^{1/2} + L(\snorm{\vec{p}}_\infty+\eta/2)^n. \qedhere\]

\end{proof}
Now we are ready to conclude \cref{prop:e1}.
\begin{proof}[Proof of \cref{prop:e1}]
A completely identical argument to the proof of \cref{prop:e1-e2} shows that for a sufficiently large constant $K$ depending on $\xi$, and for $\Xi := \on{supp(\xi)} \subseteq \mb{R}$,
\[\mb{P}[\mc{E}]\le p_0^{3n} + \sum_{a\in\mc{C}}\mb{P}[M^{(1)}=a]\mb{P}[\exists u'\in\mb{R}^{n-1}: \snorm{M^{(-1)}u'+a}_2\le 2t],\]
where $M^{(1)}$ denotes the first column of $M_{n}$, $M^{(-1)}$ denotes the $n\times (n-1)$ matrix formed by excluding this column, and 
\[\mc{C} = \{a\in\Xi^n: \exists\lambda\in\mb{R}\text{ with }\snorm{a-\lambda 1_n}_2\le 8K\delta'\sqrt{n}\}.\]
We want to bound the contribution of the sum on the right hand side. 

\textbf{Case I: }If $a = 0$, $\mb{P}[M^{(1)} = a] = p_0^{n}$.

\textbf{Case II: }If $a \neq 0$, then $\snorm{a}_{2} \ge h_{\xi} > 0$. Hence, by \cref{lem:sharp-block}, there is a constant $C > 0$ depending on $\xi, \eta$ such that
\[\mb{P}[\exists u'\in\mb{R}^{n-1}: \snorm{M^{(-1)}u'+a}_2\le 2t] \le C n^{3} \sqrt{t} + (\snorm{\vec{p}}_{\infty} + \eta/2)^{n}.\]
Moreover, a similar (but easier) argument as in the proof of \cref{prop:e1-e2} shows that
\[\mb{P}[M^{(1)} \in \mc{C}] \le \snorm{\vec{p}}_{\infty}^{n}\exp(c_{\xi, \delta'}n),\]
where $c_{\xi, \delta'}$ goes to $0$ as $\delta'$ goes to $0$. 

Hence, we see that the contribution to the sum from this case is bounded by
\[\snorm{\vec{p}}_{\infty}^{n}\exp(c_{\xi, \delta'}n)\cdot \bigg(C n^{3} \sqrt{t} + (\snorm{\vec{p}}_{\infty} + \eta/2)^{n}\bigg).\]
By choosing $\delta'$ sufficiently small depending on $\xi$ and $\eta$, and using $\snorm{\vec{p}}_{\infty}^{2} \le \snorm{\vec{p}}_{2}^{2} - c_\xi$ for some $c_\xi > 0$, we see as before (using the AM-GM inequality) that the above quantity is at most
\[t\exp(-\eta' n)+(\snorm{\vec{p}}_{2}^{2} - \eta')^{n}\]
for a sufficiently small $\eta'$ depending on $\xi$ and $\eta$. This completes the proof. 
\end{proof}

The proof of \cref{thm:compressible} is now immediate. 
\begin{proof}[Proof of \cref{thm:compressible}]
First, assume that $\xi$ is not a shift of an origin-symmetric random variable. We choose $\delta'$ small enough so that the conclusions of \cref{prop:e1,prop:e1-e2} are satisfied. By the union bound, this shows that the contribution of $\on{Elem}(\delta')$ to the probability is at most
\[n\mb{P}[\mc{E}_{e_1}] + \binom{n}{2}\mb{P}[\mc{E}_{e_1 - e_2}] + (t + \snorm{\vec{p}}_{2}^{2n})e^{-\eta n},\]
for a sufficiently small $\eta > 0$ depending on $\xi$, and for all sufficiently large $n$ depending on $\xi$. Now, we can conclude using \cref{cor:non-elementary}.

Next, if $\xi$ is a nonzero shift of an origin-symmetric random variable, we do the same, except we require \cref{prop:e1,prop:e1-e2,prop:e1+e2} and then conclude with \cref{cor:non-elementary}.

Finally, we consider the case when $\xi$ is an origin-symmetric random variable. As before, we begin by using \cref{prop:e1,prop:e1-e2}. The only thing to note is that, by the symmetry of $\xi$ about the origin, for all $i\neq j$, $\mb{P}[\mc{E}_{e_i - e_j}] = \mb{P}[\mc{E}_{e_i + e_j}]$. Hence, by the union bound, the contribution of $\on{Elem}'(\delta')$ to the probability is at most 
\[n\mb{P}[\mc{E}_{e_1}] + \binom{n}{2}(\mb{P}[\mc{E}_{e_1 - e_2}] + \mb{P}[\mc{E}_{e_1 + e_2}]) + (t+\snorm{\vec{p}}_{2}^{2n})e^{-\eta n}.\]
Now, we can conclude using \cref{cor:non-elementary}.
\end{proof}

\section{Deduction of \texorpdfstring{\cref{thm:main-approximate,thm:main-exact}}{Theorems 1.2 and 1.4}}
\label{sec:deduction-main}
Given the results in \cref{sec:incompressible,sec:compressible}, the deduction of  \cref{thm:main-approximate,thm:main-exact} is immediate. Fix a discrete distribution $\xi$, and let  $\delta,\rho,\eta, n_0 > 0$ be parameters depending on $\xi$ coming from \cref{thm:compressible}. Then, for the proof of \cref{thm:main-approximate}, let $\epsilon > 0$ be as in the statement of the theorem (it suffices to assume that $\epsilon < 1$), and for the proof of \cref{thm:main-exact}, let $\epsilon > 0$ be such that
\[\exp(2\epsilon-H(\vec{p})) < \snorm{\vec{p}}_2^2,\]
which is possible since, by the weighted AM-GM inequality, we have 
\[\exp(-H(\vec{p})) = \prod_ip_i^{p_i}\le\sum_ip_i^2 = \snorm{\vec{p}}_2^2,\]
and equality holds if and only if $\xi$ is uniform on its support. 

Let $C = C_{\ref{thm:incompressible}}(\xi, \delta, \rho, \epsilon/2)$. By taking $C_{\xi, \epsilon}$ in \cref{thm:main-approximate} and $C_{\xi}$ in \cref{thm:main-exact} to be at least $C$, we may restrict our attention to $0\le t\le 1/C$ (since for $t\ge 1/C$, the right-hand sides of \cref{thm:main-approximate,thm:main-exact} are at least $1$). By \cref{thm:compressible} and \cref{thm:incompressible}, for all $0\le t\le 1/C$, we have
\begin{align*}
\mb{P}[s_n(M_n(\xi))\le t/\sqrt{n}]&\le\mb{P}\bigg[\inf_{x\in\on{Cons}(\delta,\rho)}\snorm{M_n(\xi)x}_2\le t/\sqrt{n}\bigg]+\mb{P}\bigg[\inf_{y\in\on{Cons}(\delta,\rho)}\snorm{yM_n(\xi)}_2\le Ct\bigg]\\
&\quad+\mb{P}\bigg[\inf_{x\in\on{Noncons}(\delta,\rho)}\snorm{M_n(\xi)x}_2\le t/\sqrt{n}\wedge\inf_{y\in\on{Cons}(\delta,\rho)}\snorm{yM_n(\xi)}_2 > Ct\bigg]\\
&\le 2n\mb{P}[\mc{E}_{e_1}] + (n^2-n)(\mb{P}[\mc{E}_{e_1-e_2}] + \mb{P}[\mc{E}_{e_1+e_2}]) + 2(Ct+\snorm{\vec{p}}_2^{2n})e^{-\eta n}\\
&\quad+Ct+\exp((\epsilon/2-H(\vec{p}))n)
\end{align*}
for all sufficiently large $n$. Here, we have used that $M_n(\xi)$ and $M_n(\xi)^\intercal$ have the same distribution.

For \cref{thm:main-exact}, we are done by our choice of $\epsilon$. 

For \cref{thm:main-approximate}, we note that by Cauchy-Schwarz (as in \cref{prop:p-characterization}), $\mb{P}[\mc{E}_{e_1 + e_2}] \le \mb{P}[\mc{E}_{e_1 - e_2}]$ and recall from above that $\exp(-H(\vec{p})) \le \snorm{\vec{p}}_{2}^{2}$. Using this, we can bound the right hand side of the above computation by
\[2Ct + 2n\mb{P}[\mc{E}_{e_1}] + 2n^{2}\exp(\epsilon n/2)\mb{P}[\mc{E}_{e_1 - e_2}].\]
The desired conclusion follows since $2n^{2}\exp(\epsilon n/2) \le (1+\epsilon)^{n}$ for all $\epsilon < 1$ and $n$ sufficiently large.

\bibliographystyle{amsplain0.bst}
\bibliography{main.bib}

\end{document}